\documentclass[reqno]{amsart}
\pdfoutput=1
\usepackage{rrmath}
\usepackage[section]{rrthm}
\usepackage[all]{xy}
\usepackage[colorlinks]{hyperref}
\hypersetup{
 linkcolor={blue},
 pdftitle={Notes\ on\ Beilinson's\ 'How\ to\ glue\ perverse\ sheaves'"},
 pdfauthor={Ryan\ Reich},
 pdfdisplaydoctitle={true},
 pdfstartview = {FitH}
}
\usepackage[lite]{amsrefs}
\SelectTips{cm}{}
\CompileMatrices

\title{Notes on Beilinson's ``How to glue perverse sheaves''}
\author{Ryan Reich}
\email{ryanr@math.harvard.edu}

\newcommand{\perv}{{}^p}
\newcommand{\Psiun}{\Psi^\text{un}}
\newcommand{\Phiun}{\Phi^\text{un}}
\newcommand{\Xiun}{\Xi^\text{un}}
\newcommand{\psiun}{\psi^\text{un}}
\newcommand{\phiun}{\phi^\text{un}}
\begin{document}

\begin{abstract}
 The titular, foundational work of Beilinson not only gives a technique for gluing perverse sheaves
 but also implicitly contains constructions of the nearby and vanishing cycles functors of perverse
 sheaves. These constructions are completely elementary and show that these functors preserve
 perversity and respect Verdier duality on perverse sheaves.  The work also defines a new, ``maximal
 extension'' functor, which is left mysterious aside from its role in the gluing theorem.  In these
 notes, we present the complete details of all of these constructions and theorems.
\end{abstract}

\pdfbookmark[1]{Title}{b:title}
\maketitle

In this paper we discuss Alexander Beilinson's ``How to glue perverse sheaves'' \cite{this} with
three goals.  The first arose from a suggestion of Dennis Gaitsgory that the author study the
construction of the unipotent nearby cycles functor $R\psiun$ which, as Beilinson observes in his
concluding remarks, is implicit in the proof of his Key Lemma 2.1.  Here, we make this construction
explicit, since it is invaluable in many contexts not necessarily involving gluing.  The second goal
is to restructure the presentation around this new perspective; in particular, we have chosen to
eliminate the two-sided limit formalism in favor of the straightforward setup indicated briefly in
\cite{BB}*{\S4.2} for D-modules.  We also emphasize this construction as a simple demonstration that
$R\psiun[-1]$ and Verdier duality $\DD$ commute, and de-emphasize its role in the gluing theorem.
Finally, we provide complete proofs; with the exception of the Key Lemma, \cite{this} provides a
complete program of proof which is not carried out in detail, making a technical understanding of
its contents more difficult given the density of ideas.  This paper originated as a learning
exercise for the author, so we hope that in its final form it will be helpful as a learning aid for
others.  We do not intend it to supplant, but merely to supplement, the original, and we are
grateful to Beilinson for his generosity in permitting this.

The author would like to offer three additional thanks: to Gaitsgory, who explained how this
beautiful construction can be understood concretely, thus providing the basis for the perspective
taken here; to Sophie Morel, for confirming the author's understanding of nearby and vanishing
cycles as presented below; and to Mark de Cataldo, for his generous contribution of time and effort
to the improvement of these notes.

In order to maintain readability, we will work with sheaves of vector spaces in the classical
topology on complex algebraic varieties, except in the second part of \ref{sec:comments}, where we
will require the field of coefficients to be algebraically closed. For the necessary modifications
to \'etale sheaves, one should consult Beilinson's paper: aside from the shift in definitions the
only change is some Tate twists.  For the D-modules case, one should read Sam Lichtenstein's
undergraduate thesis, \cite{sam}, in which the two-sided limit construction is also given in detail.

\section{Theoretical preliminaries}
The topic at hand is perverse sheaves and nearby cycles; for greater accessibility of these notes,
we give a summary of the definitions and necessary properties here.

\subsection*{Diagram chases}
Occasionally, we indicate diagram chases in a proof.  For ease of reading we have tried not to make
this an essential point, but in case the reader should find such a chase to be a convincing informal
argument, we indicate here why it is also a convincing formal one.

Every object in an abelian category $\cat{A}$ can be considered, via Yoneda's lemma, to be a
sheaf, namely its functor of points, on the \emph{canonical topology} of $\cat{A}$.  This is, by
definition, the largest Grothendieck topology in which all representable functors
$\on{Hom}_{\cat{A}}(\farg, x)$ are sheaves, and its open covers are precisely the \emph{universal
strict epimorphisms}.  Such a map is, in a more general category, a map $\map{f}{u}{x}$ such that
the fibered product $u' = u \times_x u$ exists, the coequalizer $x' = \on{coker}(u'
\rightrightarrows u)$ exists, the natural map $x' \to x$ is an isomorphism, and that all of this is
\emph{also} true when we make any base change along a map $\map{g}{y}{x}$, for the induced map
$\map{f \times_x \id}{u \times_x y}{y}$.  In an abelian category, however, this is all equivalent
merely to the statement that $f$ is a surjection.

Recall the definitions of the various constructions on sheaves:
\begin{enumerate}
 \item Kernels of maps are taken sectionwise; i.e.\ for a map $\map{f}{\sh{F}}{\sh{G}}$,
 $\on{ker}(f)(U) = \on{ker}(\map{f(U)}{\sh{F}(U)}{\sh{G}(U)})$.  Likewise, products and limits are
 taken sectionwise.

 \item Cokernels are \emph{locally} taken sectionwise: any section $s \in \on{coker}(f)(U)$ is, on
 some open cover $V$ of $U$, of the form $\bar{t}$ for $t \in \sh{G}(V)$.  Likewise, images,
 coproducts, and colimits are taken locally.
\end{enumerate}
In an abelian category, where all of these constructions exist by assumption, these descriptions are
even prescriptive: if one forms the sheaves thus described, they are representable by the objects
claimed.  Therefore, the following common arguments in diagram chasing are valid:
\begin{enumerate}
 \item A map $\map{f}{x}{y}$ is surjective if and only if for every $s \in y$, there is some $t \in
 x$ such that $s = f(t)$.  This is code for: for every ``open set'' $U$ and every $s \in y(U)$,
 there is a surjection $V \to U$ and a section $t \in y(V)$ such that $s|_V = f(t)$.
 
 \item If $s \in y$, then $\bar{s} = 0 \in \on{coker}(f)$ if and only if $s \in \on{im}(f)$.  This
 is code for: if $s \in y(U)$ and $\bar{s} = 0 \in \on{coker}(f)(U)$, then there is some surjection
 $V \to U$ and $t \in x(V)$ with $s|_V = f(t)$.
 
 \item For $s, t \in x$, $s = t$ if and only if $s - t = 0$.  Here, the sum of maps $\map{s}{U}{x}$
 and $\map{t}{V}{x}$ is obtained by forming the fibered product $W = U \times_x V$ which covers both
 $U$ and $V$, and then taking the sum of the maps $s|_W, t|_W \in \on{Hom}(W,x)$; the condition for
 equality is just the statement that a section of a sheaf vanishes if only it vanishes on an open
 cover.
\end{enumerate}
Any other arguments involving elements and some concept related to exactness can also be phrased in
this way.  Thus, a na\"ive diagram-chasing argument can be converted into a rigorous one simply by
replacing statements like $s \in x$ with correct ones $s \in x(U)$ for some open set $U$, and
passing to surjective covers of $U$ when necessary.

\subsection*{Derived category and functors}
All the action takes place in the derived category; specifically, let $X$ be an algebraic variety
and denote by $\cat{D}(X)$ its derived category of bounded complexes of sheaves of vector spaces
with constructible cohomology.  By definition, a map of complexes $\map{f}{A^\bullet}{B^\bullet}$
defines an isomorphism in $\cat{D}(X)$ if and only if its associated map on cohomology sheaves
$\map{H^i(f)}{H^i(A^\bullet)}{H^i(B^\bullet)}$ is an isomorphism for all $i$.  We have a notation
for the index-shift: $A^{i + 1} = (A[1])^i$ (technically, the differential maps also change sign,
but we will never need to think about this).  The derived category $\cat{D}(X)$ is a ``triangulated
category'', which means merely that in it are a class of triples, called ``distinguished
triangles'', of complexes and maps
\begin{equation*}
 A^\bullet \to B^\bullet \to C^\bullet \to A^\bullet[1]
\end{equation*}
in which two consecutive arrows compose to zero, satisfying the axioms given in, for example,
\cite{GM} (but see also \ref{sec:comments}), and with the property that the associated long
sequence of cohomology sheaves
\begin{equation*}
 \dots H^{-1}(C^\bullet) \to H^0(A^\bullet) \to H^0(B^\bullet) \to H^0(C^\bullet) \to H^1(A^\bullet)
\to \dots
\end{equation*}
is exact (note that $H^0(A^\bullet [1]) = H^1(A^\bullet)$); we say that the $H^i$ are
``cohomological''.  If $\map{f}{A^\bullet}{B^\bullet}$ is given, there always exists a triangle
whose third term $C^\bullet = \on{Cone}(f)$ is the ``cone'' of $f$; this cone is unique up to
nonunique isomorphism and any commutative diagram of maps $f$ induces a map on cones, but this
is not functorial.  It follows that the induced triangle itself is unique up to a nonunique
isomorphism whose component morphisms on $A^\bullet$ and $B^\bullet$ are the identity maps.  A
functor between two triangulated categories is ``triangulated'' if it sends triangles in one to
triangles in the other.

In $\cat{D}(X)$ we also have some standard constructions of sheaf theory.  For any two complexes
there is the ``total tensor product'' $A^\bullet \tensor B^\bullet$ obtained by taking in degree $n$
the direct sum of all products $A^i \tensor B^j$ with $i + j = n$ (and some differentials that are
irrelevant) and its derived bifunctor $A^\bullet \tensor^L B^\bullet$, with $H^i(A^\bullet \tensor^L
B^\bullet) = \on{Tor}^i(A^\bullet, B^\bullet)$, which is a triangulated functor in each variable. We
also have the bifunctor (contravariant in the first argument) $\shHom(A^\bullet, B^\bullet)$, whose
terms are $\shHom(A^\bullet, B^\bullet)^i(U) = \on{Hom}(A^\bullet|_U, B^\bullet[i]|_U)$, and its
derived bifunctor $R\shHom(A^\bullet, B^\bullet)$, with $H^i R\shHom(A^\bullet, B^\bullet) =
\on{Ext}^i(A^\bullet, B^\bullet)$, which is triangulated in each variable. Of course, these two have
an adjunction:
\begin{equation*}
 R\shHom(A^\bullet \tensor^L B^\bullet, C^\bullet) \cong R\shHom(A^\bullet, R\shHom(B^\bullet,
C^\bullet)).
\end{equation*}

For any Zariski-open subset $U \subset X$ with inclusion map $j$, there are triangulated functors
$\map{j_!, j_*}{\cat{D}(U)}{\cat{D}(X)}$ and $\map{j^* = j^!}{\cat{D}(X)}{\cat{D}(U)}$; if $i$ is
the inclusion of its complement $Z$, then there are likewise maps $\map{i^!,
i^*}{\cat{D}(X)}{\cat{D}(Z)}$ and $\map{i_* = i_!}{\cat{D}(Z)}{\cat{D}(X)}$.  (Technically
the operation $j_*$ is only left exact on sheaves and we should write $Rj_*$ for its derived
functor, but we will never have occasion to use the plain version so we elide this extra notation.) 
They satisfy a number of important relations, of which we will only use one here: there is a
functorial triangle in the complex $A^\bullet_X \in \cat{D}(X)$:
\begin{equation}
 \label{eq:extension triangle}
 j_!j^*(A^\bullet_X) \to A^\bullet_X \to i_* i^* (A^\bullet_X) \to
\end{equation}
We will generally forget about writing $i_*$ and consider $\cat{D}(Z) \subset \cat{D}(X)$.

There is also a triangulated duality functor $\map{\DD}{\cat{D}(X)}{\cat{D}(X)^\op}$ which
interchanges $!$ and $*$, in that $\DD j_* (A^\bullet_U) = j_!(\DD A^\bullet_U)$, etc., and is an
involution.  In fact, if we set $\sh{D}^\bullet_X = \DD \csheaf \C$, then $\DD(A^\bullet) =
R\shHom(A^\bullet, \sh{D}^\bullet_X)$.

For any map $\map{f}{X}{Y}$ of varieties, we have $f^!, f^*$ as well (also $f_!, f_*$, and none of
them are equal), with the same relationships to $\DD$, and the useful identity
\begin{equation}
 \label{eq:exceptional hom localization}
 f^! R\shHom(A^\bullet_Y, B^\bullet_Y) = R\shHom(f^* A^\bullet_Y, f^! B^\bullet_Y).
\end{equation}
Note that by these properties, we have $f^!\sh{D}^\bullet_Y = \sh{D}^\bullet_X$.

\subsection*{Perverse sheaves}
Here we give a detail-free overview of the formalism of perverse sheaves created in \cite{BBD}.
Within $\cat{D}(X)$ there is an abelian category $\cat{M}(X)$ of ``perverse sheaves'' which has
nicer properties than the category of actual sheaves.  It is specified by means of a
``t-structure'', namely, a pair of full subcategories $\perv \cat{D}(X)^{\leqslant 0}$ and $\perv
\cat{D}(X)^{\geqslant 0}$, also satisfying some conditions we won't use, and such that
\begin{equation*}
 \cat{M}(X) = \perv\cat{D}(X)^{\leqslant 0} \cap \perv\cat{D}(X)^{\geqslant 0}.
\end{equation*}
There are truncation functors $\map{\tau^{\leqslant 0}}{\cat{D}(X)}{\perv\cat{D}(X)^{\leqslant 0}}$
and
likewise for $\tau^{\geqslant 0}$, fitting into a distinguished triangle for any complex
$A^\bullet_X \in \cat{D}(X)$:
\begin{equation*}
 \tau^{\leqslant 0} A^\bullet_X \to A^\bullet_X \to \tau^{> 0} A^\bullet_X \to
\end{equation*}
(where $\tau^{> 0} = \tau^{\geqslant 1} = [-1] \circ \tau^{\geqslant 0} \circ [1]$).  This triangle
is \emph{unique} with respect to the property that the first term is in $\perv\cat{D}(X)^{\leqslant
0}$ and the third is in $\perv\cat{D}(X)^{> 0}$.  They have the obvious properties implied by the
notation: $\tau^{\leqslant a} \tau^{\leqslant b} = \tau^{\leqslant a}$ if $a \leq b$, and likewise
for $\tau^{\geqslant ?}$. Furthermore, there are ``perverse cohomology'' functors $\map{\perv
H^i}{\cat{D}(X)}{\cat{M}(X)}$, where of course $\perv H^i(A^\bullet) = \perv H^0(A^\bullet[i])$ and
$\perv H^0 = \tau^{\geqslant 0} \tau^{\leqslant 0} = \tau^{\leqslant 0} \tau^{\geqslant 0}$; these
are cohomological just like the ordinary cohomology functors. The abelian category structure of
$\cat{M}(X)$ is more or less determined by the fact that if we have a map $\map{f}{\sh{F}}{\sh{G}}$
of perverse sheaves (this is the notation we will be using; we will not think of perverse sheaves as
complexes), then
\begin{align*}
 \on{ker} f = \perv H^{-1} \on{Cone}(f) && \on{coker} f = \perv H^0 \on{Cone}(f).
\end{align*}
For notational convenience, we will write $\sh{M}$ for a perverse sheaf on $U$, $\sh{F}$ for one on
$X$, and as usual, abandon $i_*$ and just consider $\cat{M}(Z) \subset \cat{M}(X)$ (for the reason
expressed immediately below, this is reasonable).

The category $\cat{M}(X)$ is closed under the duality functor $\DD$, but not necessarily under the
six functors defined for an open/closed pair of subvarieties.  However, it is true that
$j_*(\sh{M}), i^!(\sh{F}) \in \perv \cat{D}^{\geqslant 0}$ and $j_!(\sh{M}), i^*(\sh{F}) \in \perv
\cat{D}^{\leqslant 0}$, while $j^*(\sh{F}), i_*(\sh{F}_Z) \in \cat{M}$ ($\sh{F}_Z$ a perverse sheaf
on $Z$); we say these functors are right, left, or just ``t-exact''.  Furthermore, when $j$ is an
affine morphism (the primary example being when $Z$ is a Cartier divisor), both $j_!$ and $j_*$ are
t-exact, and thus their restriction to $\cat{M}(U)$ is exact with values in $\cat{M}(X)$.  There is
also a ``minimal extension'' functor $j_{!*}$, defined so that $j_{!*}(\sh{M})$ is the image of
$\perv H^0(j_! \sh{M})$ in $\perv H^0(j_* \sh{M})$ along the natural map $j_! \to j_*$; it is the
unique perverse sheaf such that $i^* j_{!*} \sh{M} \in \perv \cat{D}^{< 0}(Z)$ and $i^! j_{!*}
\sh{M} \in \perv \cat{D}^{> 0}(Z)$, but for us the most useful property is that when $j$ is an
affine, open immersion, then we have a sequence of \emph{perverse sheaves}
\begin{equation}
 \label{eq:kernel and cokernel}
 i^* j_{!*} \sh{M} [-1] \incl j_! \sh{M} \surj j_{!*} \sh{M} \incl j_* \sh{M}
  \surj i^! j_{!*}  \sh{M}[1];
\end{equation}
i.e.\ $i^* j_{!*} \sh{M} [-1] = \on{ker}(j_! \sh{M} \to j_* \sh{M})$ and $i^! j_{!*} \sh{M}[1] =
\on{coker}(j_! \sh{M} \to j_* \sh{M})$ are both perverse sheaves.

Perverse sheaves have good category-theoretic properties: $\cat{M}(X)$ is both artinian and
noetherian, so every perverse sheaf has finite length.  Finally, we will use the sheaf-theoretic
fact that if $\sh{L}$ is a locally constant sheaf on $X$, then $\sh{F} \tensor \sh{L}$ is perverse
whenever $\sh{F}$ is.  Note that since $\sh{L}$ is locally free, it is flat, and therefore $\sh{F}
\tensor \sh{L} = \sh{F} \tensor^L \sh{L}$.

\subsection*{Nearby cycles}
If we have a map $\map{f}{X}{\Aff^1}$ such that $Z = f^{-1}(0)$ (so $U = f^{-1}(\Aff^1 \setminus
\{0\}) = f^{-1}(\Gm)$), the ``nearby cycles'' functor $\map{R\psi_f}{\cat{D}(U)}{\cat{D}(Z)}$ is
defined.  Namely, let $\map{u}{\tilde{\Gm}}{\Gm}$ be the universal cover of $\Gm = \Aff^1 \setminus
\{0\}$, let $\map{v}{\tilde{U} = U \times_{\Gm} \tilde{\Gm}}{U}$ be its pullback, forming a diagram
\begin{equation*}
 \xymatrix@R-2em{
                     &             &                       & {\tilde{U}} \ar[dl]_{v} \ar[dd] \\
  Z \ar[r]^i \ar[dd] & X \ar[dd]^f & U \ar[l]_{j} \ar[dd] &                                  \\
                     &             &                       & {\tilde{\Gm}} \ar[dl]^{u}       \\
  \{0\} \ar[r]       & \Aff^1      & \Gm \ar[l]            & 
 }
\end{equation*}
and set (in this one instance, explicitly writing $j_*$ and $v_*$ as non-derived functors)
\begin{equation*}
 R\psi_f = R(i^* j_* v_* v^*) \colon \cat{D}(U) \to \cat{D}(Z).
\end{equation*}
Since $i^*$ and $v^*$ are exact, indeed $\psi_f$ is a left-exact functor from sheaves on $U$ to
sheaves on $Z$.  Many sources (e.g.\ \cite{schurmann}*{\S1.1.1}) give the definition $R\psi_f = i^*
Rj_* Rv_* v^*$; in fact, they are the same: since $v$ is a covering map, if $\sh{F}$ is a
flasque sheaf on $U$, then $v^* \sh{F}$ is flasque on $\tilde{U}$ and so acyclic for $v_*$ (and $v_*
v^* \sh{F}$ acyclic for $j_*$). Therefore we may form the derived functor before or after
composition. Note that $v$ is not an algebraic map, and therefore it is not \textit{a priori} clear
whether $R\psi_f$ preserves constructibility; that it does is a theorem of Deligne (\cite{SGA},
Expos\'e XIII, Th\'eor\`eme 2.3 for \'etale sheaves and Expos\'e XIV, Th\'eor\`eme 2.8 for the
comparison with classical nearby cycles).

The fundamental group $\pi_1(\Gm)$ acts on any $v^* A^\bullet_U$ via deck transformations of
$\tilde{\Gm}$ and therefore acts on $\psi_f$ and $R\psi_f$.  There is a natural map $i^* A^\bullet_X
\to \psi_f(j^* A^\bullet_X)$, obtained from $(v^*, v_*)$-adjunction, on whose image $\pi_1(\Gm)$
acts trivially. We set, by definition,
\begin{equation*}
 i^* A^\bullet_X \to \psi_f(j^* A^\bullet_X) \to \phi_f(A^\bullet_X) \to 0
\end{equation*}
where $\phi_f(A^\bullet_X)$ is the ``vanishing cycles'' sheaf.  Using some homological algebra
tricks the above sequence induces a natural distinguished triangle
\begin{equation*}
 i^* A^\bullet_X \to R\psi_f(j^* A^\bullet_X) \to R\phi_f(A^\bullet_X) \to
\end{equation*}
where $R\phi_f$ is (morally) the right derived functor of $\phi_f$.  Like $R\psi_f$, $R\phi_f$ has
a monodromy action of $\pi_1(\Gm)$; this action is one of the maps on the cone of the above
triangle induced by the monodromy action on $R\psi_f$, but as this is not functorial, one should
consult the real definition in \cite{SGA} (given for the algebraic nearby cycles, but see also the
second expos\'e).

\begin{theorem}{lem}{unipotent nearby cycles}
 There exists a unique decomposition of $R\psi_f$ as $R\psiun_f \oplus R\psi_f^{\neq 1}$, where for
 any choice of generator $t$ of $\pi_1(\Gm)$, $1 - t$ acts nilpotently on $R\psiun_f(A^\bullet_U)$
 for any complex $A^\bullet_U$ and is an automorphism of $R\psi_f^{\neq 1}$.
\end{theorem}

The part $R\psiun_f$ is called the functor of \emph{unipotent nearby cycles}.

\begin{proof}
 To start, we observe that for any sheaf $\sh{F}$ on $U$, we have $\psi_f(\sh{F}) = H^0
 R\psi_f(\sh{F})$, and so if $\sh{F}$ is constructible, by the constructibility of nearby cycles so
 is $\psi_f(\sh{F})$; thus, for any open $V \subset X$, $\psi_f(\sh{F})(V)$ is finite-dimensional. 
 Let $\psiun_f \subset \psi_f$ be the subfunctor such that for any sheaf $\sh{F}$ on $U$,
 $\psiun_f(\sh{F})$ is the subsheaf of $\psi_f(\sh{F})$ in which $1 - t$ is nilpotent, so for each
 $V$, $\psiun_f(\sh{F})(V)$ is the generalized eigenspace of $t$ with eigenvalue $1$.

 Therefore it is actually a direct summand; we recall the general argument which works over any
 field $k$.  If $T$ is an endomorphism of a finite-dimensional vector space $M$, we view $M$ as a
 $k[x]$-module with $x$ acting as $T$. By the classification of modules over a principal ideal
 domain, we have $M \cong \bigoplus k[x]/p(x)$ for certain polynomials $p(x)$. The generalized
 eigenspace with eigenvalue $1$ is then the sum of those pieces for which $p(x)$ is a power of $1 -
 x$, and the remaining summands are a $T$-invariant complement in which $1 - T$ acts invertibly. As
 the image of $(1 - T)^n$ ($n \gg 0$), this complement is functorial in the category of
 finite-dimensional $k[x]$-modules and so we have the same decomposition in sheaves of
 finite-dimensional $k[x]$-modules.
 
 Specializing again to the present situation, we get a decomposition $\psi_f(\sh{F}) \cong
 \psiun_f(\sh{F}) \oplus \psi_f^{\neq 1}$ (this is the definition of $\psi_f^{\neq 1}$).  Both
 summands are \textit{a fortiori} left-exact functors and taking derived functors, we obtain a
 decomposition:
 \begin{equation*}
  R\psi_f \cong R\psiun_f \oplus R\psi^{\neq 1}_f.
 \end{equation*}
 Since $1 - t$ is nilpotent on any $\psiun_f(\sh{F})$, we may apply $\psiun_f$ to any constructible
 complex of injectives, thus computing $R\psiun_f(A^\bullet_U)$ for any complex $A^\bullet_U$, and
 conclude that $1 - t$ is nilpotent on each such; by general principles it is invertible on
 $R\psi_f^{\neq 1}$. This is what we want.
 
 Uniqueness of the decomposition is clear; indeed, if in any category with a zero object we have
 objects $x$ and $y$ together with endomorphisms $N$ and $I$ respectively such that $N$ is nilpotent
 and $I$ invertible, then any map $\map{g}{x}{y}$ intertwining $N$ and $I$ is zero: we have $gN =
 Ig$, so $g = I^{-1} g N = I^{-2} g N^2 = \dots = I^{-n} g N^n = 0$ if $N^n = 0$.  For morphisms $y
 \to x$ we work in the opposite category. In particular, if we have \begin{equation*} R\psi_f \cong
 F \oplus G \end{equation*} as a sum of two functors as in the statement of the lemma, then the
 identity map on $R\psi_f$ has no $G$-component on $R\psiun_f$ and no $F$-component on
 $R\psi_f^{\neq 1}$, and so induces isomorphisms $R\psiun_f \cong F$ and $R\psi_f^{\neq 1} \cong G$.
\end{proof}

We note that this lemma is a special case of \ref{nearby cycles decomposition} when the field of
coefficients is algebraically closed.  However, this decomposition is defined over any field.

There is a triangle, functorial in $A^\bullet_X$,
\begin{equation*}
 i^* j_* j^* A^\bullet_X \to R\psi_f(j^* A^\bullet_X) \xrightarrow{1 - t} R\psi_f(j^* A^\bullet_X)
\end{equation*}
(see \citelist{\cite{brylinski}*{Prop. 1.1} \cite{schurmann}*{eq. (5.88)}}) which, taking
$A^\bullet_U = j^* A^\bullet_X$ and inserting $R\psiun_f$ because the monodromy acts trivially on
the first term, gives the extremely important (for us) triangle
\begin{equation}
 \label{eq:psiun}
 i^* j_* A^\bullet_U \to R\psiun_f(A^\bullet_U) \xrightarrow{1 - t} R\psiun_f(A^\bullet_U) \to
\end{equation}
We also have a unipotent part of the vanishing cycles functor $R\phi_f$, and, again since the
monodromy acts trivially on $i^* A^\bullet_X$, a corresponding triangle
\begin{equation}
 \label{eq:phiun}
 i^* A^\bullet_X \to R\psiun_f(A^\bullet_X) \to R\phiun_f(A^\bullet_X) \to
\end{equation}
If $\sh{L}$ is any locally constant sheaf on $\Gm$ with underlying vector space $L$ and unipotent
monodromy, then $R\psiun_f(A^\bullet_U \tensor f^* \sh{L}) \cong R\psiun_f(A^\bullet_U) \tensor L$,
where $\pi_1(\Gm)$ acts on the tensor product by acting on each factor (since $\sh{L}$ is
trivialized on $\tilde{\Gm}$).

We note the following fact, crucial to all computations in this paper:
\begin{center}
 \emph{$j$ is an affine morphism.}
\end{center}
Indeed, $Z$ is cut out by a single algebraic equation.
Although when $Z$ is any Cartier divisor it is still locally defined by equations $f$ and the
inclusion $j$ of its complement is again an affine morphism, it is not necessarily possible to glue
nearby cycles which are locally defined as above; c.f.\ \cite{survey}*{Remark 5.5.4}; an explicit
example will appear in \cite{MTZ}.  However, it follows from \ref{nearby cycles computations} that
when $\sh{M}$ is a perverse sheaf and $R\psiun_f(\sh{M})$ has trivial monodromy, it is in fact
independent of $f$ and gluing is indeed possible.


Triangle \ref{eq:psiun} already implies that nearby cycles preserve perverse sheaves.

\begin{theorem}{lem}{nearby cycles are perverse}
 The functor $R\psiun_f[-1]$ sends $\perv \cat{D}(U)^{\leqslant 0}$ to $\perv \cat{D}(Z)^{\leqslant
 0}$ and takes $\cat{M}(U)$ to $\cat{M}(Z)$.
\end{theorem}

\begin{proof}
 Since $j$ is affine and an open immersion, $j_*$ and $j_!$ are t-exact, so for any $A^\bullet_U \in
 \perv \cat{D}(U)^{\leqslant 0}$, $i^* j_* A^\bullet_U = \on{Cone}(j_! A^\bullet_U \to j_*
 A^\bullet_U)$ is in $\perv \cat{D}(Z)^{\leqslant 0}$.  If we apply the long exact sequence of
 perverse cohomology to triangle \ref{eq:psiun}, we therefore get in nonnegative degrees:
 \begin{multline*}
  \perv H^0(R\psiun_f A^\bullet_U) \xrightarrow{1 - t} \perv H^0(R\psiun_f A^\bullet_U)
        \to (0 = \perv H^1(i^* j_* A^\bullet_U)) \to \\
  \perv H^1(R\psiun_f A^\bullet_U) \xrightarrow{1 - t} \perv H^1(R\psiun_f A^\bullet_U)
        \to (0 = \perv H^2(i^* j_* A^\bullet_U)) \to \cdots
 \end{multline*}
 For $i \geq 0$, the map $\perv H^i(R\psiun_f A^\bullet_U) \to \perv H^i(R\psiun_f A^\bullet_U)$ is
 both given by a \emph{nilpotent} operator and is surjective, so zero. It follows that
 $R\psiun_f(A^\bullet_U) \in \perv \cat{D}(Z)^{\leqslant -1}$, as promised.

 Now let $\sh{M} \in \cat{M}(U)$ be a perverse sheaf.  Then $i^* j_* \sh{M} \in \perv
 \cat{D}(Z)^{[-1,0]}$ since its perverse cohomology sheaves are the kernel and cokernel of the map
 $j_! \sh{M} \to j_* \sh{M}$.  In degrees $\leq -2$, then, we have
 \begin{multline*}
 \dots \to (0 = \perv H^{-3}(i^* j_* \sh{M})) \to
       \perv H^{-3}(R\psiun_f \sh{M}) \xrightarrow{1 - t} \perv H^{-3}(R\psiun_f \sh{M}) \to \\
  (0 = \perv H^{-2}(i^* j_* \sh{M})) \to
         \perv H^{-2}(R\psiun_f \sh{M}) \xrightarrow{1 - t} \perv H^{-2}(R\psiun_f \sh{M})
 \end{multline*}
 This means that for $i \leq -2$, all the maps $1 - t$ are injective and nilpotent, hence zero.
 Thus $R\psiun_f(\sh{M}) \in \perv \cat{D}(Z)^{-1}$, as desired.
\end{proof}

Since $R\psiun_f[-1]$ acts on perverse sheaves, we will give it the abbreviated notation $\Psiun_f$.

\section{Construction of the unipotent nearby cycles functor}
\label{sec:unipotent cycles}

Let $L^a$ be the vector space of dimension $a \geq 0$ together with the action of a matrix $J^a =
[\delta_{ij} - \delta_{i, j - 1}]$, a unipotent (variant of a) Jordan block of dimension $a$.  Let
$\sh{L}^a$ be the locally constant sheaf on $\Gm$ whose underlying space is $L^a$ and in whose
monodromy action a (hereafter fixed choice of) generator $t$ of $\pi_1(\Gm)$ acts by $J^a$.  Since
it is locally free, it is flat, so we will write $\tensor$ rather than $\tensor^L$ in tensor
products with it (actually, with $f^* \sh{L}^a$).  It has the following self-duality properties,
where $\check{\sh{L}}^a = \shHom(\sh{L}^a, \csheaf{\C})$ is the dual local system and
$(\sh{L}^a)^{-1}$ is the local system in whose monodromy $t$ acts by $(J^a)^{-1}$:

\begin{theorem}{lem}{jordan block dual}
 We have $\sh{L}^a \cong \check{\sh{L}}^a \cong (\sh{L}^a)^{-1}$, and $\DD(A^\bullet_U
 \tensor f^* \sh{L}^a) \cong \DD(A^\bullet_U) \tensor f^* \sh{L}^a$ for $A^\bullet_U \in
 \cat{D}(U)$.
\end{theorem}

\begin{proof}
 Since $\check{\sh{L}}^a$ is the local system with vector space the dual $\check{L}^a$ and monodromy
 $((J^a)^t)^{-1}$, its monodromy is again unipotent with a single Jordan block of length $a$.  We
 fix in $L^a$ the given basis $\vect{e}_1, \dots, \vect{e}_a$ associated to $J$, and in
 $\check{L}^a$ we choose a generalized eigenbasis $\check{f}_1, \dots, \check{f}_n$ in which
 $((J^a)^t)^{-1}$ has the matrix $J^a$, so that sending $\vect{e}_i \mapsto \check{f}_i$ identifies
 $J^a$ with $((J^a)^t)^{-1}$ and thus induces the desired map of local systems.  The same proof
 shows that $\sh{L}^a \cong (\sh{L}^a)^{-1}$.

 In general, then, we construct an isomorphism:
 \begin{equation}
  \label{eq:dual iso}
  \DD(A^\bullet_U) \tensor f^* \sh{L}^a \xrightarrow{\sim}
  \DD(A^\bullet_U \tensor f^* \sh{L}^a),
 \end{equation}
 where
 \begin{align*}
  \DD(A^\bullet_U) \tensor f^* \sh{L}^a
   &= \smash{R\shHom(A^\bullet_U, \sh{D}^\bullet) \tensor^L f^* \sh{L}^a},
  \\
  \DD(A^\bullet_U \tensor f^* \sh{L}^a)
   &= R\shHom(A^\bullet_U \tensor^L f^*\sh{L}^a, \sh{D}^\bullet_U)
   = R\shHom(A^\bullet_U, \DD f^*\sh{L}^a),
 \end{align*}
 by constructing a map
 \begin{align*}
  \smash{R\shHom(A^\bullet_U, \sh{D}^\bullet_U) \tensor^L f^* \sh{L}^a} \to
  R\shHom(A^\bullet_U, \DD f^* \sh{L}^a).
 \end{align*}
 Such a map can be obtained by applying $(\tensor^L, R\shHom)$-adjunction to a map:
 \begin{multline}
  \label{eq:dual iso adjoint}
  R\shHom(A^\bullet_U, \sh{D}_U^\bullet)
    \to R\shHom(f^* \sh{L}^a, R\shHom(A^\bullet_U, \DD f^* \sh{L}^a)) \\
    = R\shHom(A^\bullet_U, R\shHom(f^* \sh{L}^a, \DD f^* \sh{L}^a)).
 \end{multline}
 Since $\DD$ exchanges $!$ and $*$, we have $\DD f^* \sh{L}^a = f^! \DD \sh{L}^a$.  By the property
 \ref{eq:exceptional hom localization} of $f^!$, we have
 \begin{equation*}
  R\shHom(f^* \sh{L}^a, f^! \DD \sh{L}^a) = f^! R\shHom(\sh{L}^a, \DD\sh{L}^a)
 \end{equation*}
 Note also that $\sh{D}_U^\bullet = f^! \sh{D}_{\Gm}^\bullet$ by definition and therefore the map
 \ref{eq:dual iso adjoint} can be constructed by applying $R\shHom(A^\bullet_U, f^! \farg)$ to a
 certain map on $\Gm$:
 \begin{equation*}
  \sh{D}_{\Gm}^\bullet \to R\shHom(\sh{L}^a, \DD\sh{L}^a) = \DD(\sh{L}^a \tensor^L \sh{L}^a).
 \end{equation*}
 This map, in turn, is obtained by first replacing the $\tensor^L$ with $\tensor$ (since $\sh{L}^a$
 is locally free) and applying $\DD$ to the pairing
 \begin{equation*}
  \sh{L}^a \tensor \sh{L}^a \to \csheaf{\C}
 \end{equation*}
 given by the isomorphism $\sh{L}^a \cong \check{\sh{L}}^a$ described in the first paragraph.
 Thus, locally \ref{eq:dual iso} is the tautological isomorphism $(\DD A^\bullet_U)^{\oplus a} \cong
 \DD (A^\bullet_U)^{\oplus a}$. Since it is a local isomorphism, it is an isomorphism.
\end{proof}

In the rest of this section, $\sh{M}$ is any object of $\cat{M}(U)$.  The following construction is
Beilinson's definition of the unipotent nearby cycles:

\begin{theorem}{prop}{nearby cycles construction}
 Let $\map{\alpha^a}{j_!(\sh{M} \tensor f^* \sh{L}^a)}{j_*(\sh{M} \tensor f^* \sh{L}^a)}$ be the
 natural map.  Then there is an inclusion $\on{ker}(\alpha^a) \incl \Psiun_f(\sh{M})$,
 identifying the actions of $\pi_1(\Gm)$, which is an isomorphism for all sufficiently large $a$.
 (In fact, it suffices to take $a$ large enough that $(1 - t)^a$ annihilates $\Psiun_f(\sh{M})$.)
\end{theorem}

\begin{proof}
 We know by \ref{nearby cycles are perverse} that $\Psiun_f(\farg)$ is a perverse sheaf, so
 taking together the triangle \ref{eq:psiun} with $A^\bullet_U = \sh{M} \tensor f^* \sh{L}^a$ and
 exact sequence \ref{eq:extension triangle} with $A^\bullet_X = j_* A^\bullet_U$, we see that
 $\on{ker} \alpha = \on{ker}(1 - t)$, where $1 - t$ is the map appearing in the former triangle
 shifted by $-1$.  We also have
 \begin{equation*}
  \Psiun_f(\sh{M} \tensor f^* \sh{L}^a)
    \cong \Psiun_f(\sh{M}) \tensor L^a
    \cong \bigoplus_{i = 1}^a \Psiun_f(\sh{M})_{(i)},
 \end{equation*}
 where the $i$'th coordinate of the action of $t$ is $t_{(i)} - t_{(i + 1)}$, with
 $t_{(i)}$ the copy of $t \in \pi_1(\Gm)$ acting on $\Psiun_f(\sh{M})$ considered as the $i$'th
 summand.  That is, using elements, $(x_1, \dots, x_n) \in \Psiun_f(\sh{M} \tensor f^*
 \sh{L}^a)$ is sent by $t$ to $(tx_1 - tx_2, tx_2 - tx_3, \dots, tx_n)$.  Thus, for an element
 of $\on{ker}(1 - t)$, we have $x_{i + 1} = (1 - t^{-1}) x_i$, or:
 \begin{align*}
  x_i = (1 - t^{-1})^{i - 1} x_1 && -t(1 - t^{-1})^a x_1 = (1 - t)x_n = 0.
 \end{align*}
 If we define a map $\map{u}{\Psiun_f(\sh{M})}{\Psiun_f(\sh{M} \tensor f^* \sh{L}^a)}$ by
 sending the element $x = x_1$ to the coordinates $x_i$ defined by the first formula above, then
 $u$ is injective and its image contains $\on{ker}(1 - t)$ (namely, that subspace
 satisfying the second equation).  Since $1 - t$ (hence $1 - t^{-1}$) is nilpotent on
 $\Psiun_f(\sh{M})$, for $a$ sufficiently large, $\on{im}(u) = \on{ker}(1 - t)$.  We claim that $u$
 intertwines the actions of $t^{-1}$ and $J^a$:
 \begin{multline*}
  J^a u(x)
   = (x - (1 - t^{-1})x, (1 - t^{-1})x + (1 - t^{-1})^2 x, \dots) \\
   = (t^{-1} x, (1 - t^{-1})t^{-1}x, \dots)
   = u(t^{-1} x).
 \end{multline*}
 Finally, we employ the isomorphism $\sh{L}^a \cong (\sh{L}^a)^{-1}$ of \ref{jordan block dual} to
 give an automorphism of $\on{ker}(\alpha^a) \subset j_!(\sh{M} \tensor f^* \sh{L}^a)$ intertwining
 $J^a$ and $(J^a)^{-1}$.
\end{proof}

\begin{theorem}{cor}{uniformity}
 There exists an integer $N$ such that $(1 - t)^N$ annihilates both $\on{ker} \alpha^a$ and
 $\on{coker} \alpha^a$ for all $a$.
\end{theorem}

\begin{proof}
 By \ref{nearby cycles construction}, the kernel is contained in $\Psiun_f(\sh{M})$ and thus
 annihilated by that power of $1 - t$ which annihilates the nearby cycles.  Temporarily let
 $\alpha^a = \alpha^a_{\sh{M}}$; then $\DD(\alpha^a_{\sh{M}}) = \alpha^a_{\DD\sh{M}}$, so
 $\on{coker}(\alpha^a_{\sh{M}}) = \DD\on{ker}(\alpha^a_{\DD\sh{M}})$ is again annihilated by some
 $(1 - t)^N$.
\end{proof}

In preparation for the next section, we give a generalization of this construction.  For each $a, b
\geq 0$ there is a natural short exact sequence
\begin{equation*}
 0 \to \sh{L}^a \xrightarrow{g^{a,b}} \sh{L}^{a + b} \xrightarrow{g^{a + b, -a}} \sh{L}^b \to 0;
\end{equation*}
that is, for any $r \in \Z$, $g^{a, r}$ sends $\sh{L}^a$ to the first $a$ coordinates of $\sh{L}^{a
+ r}$ if $r \geq 0$, and to the quotient $\sh{L}^{a - (-r)}$ given by collapsing the first $-r$
coordinates if $-r \geq 0$ (that is, $r \leq 0$) and $a + r \geq 0$.  This sequence respects the
action of $\pi_1(\Gm)$ on the terms and, via \ref{jordan block dual}, the $(a,b)$ sequence is dual
to the $(b, a)$ sequence.

Let $\sh{M} \in \cat{M}(U)$; then we have induced maps on the tensor products:
\begin{equation*}
 \map{g^{a,r}_{\sh{M}} = \id \tensor g^{a,r}}
     {\sh{M} \tensor f^* \sh{L}^a}{\sh{M} \tensor f^* \sh{L}^{a + r}}
\end{equation*}
(we will often omit the subscript $\sh{M}$ when no confusion is possible).  By \ref{jordan block
dual}, these satisfy
\begin{equation}
 \label{eq:g duality}
 \DD g^{a,r}_{\sh{M}} = g^{a + r, -r}_{\DD \sh{M}}.
\end{equation}
Note that since the $\sh{L}^a$ are locally free, the $g^{a,r}_{\sh{M}}$ are all injective when $r
\geq 0$ and surjective when $r \leq 0$.  Let $r \in \Z$ and set 
\begin{equation*}
 \alpha^{a,r}
  = j_*(g^{a,r}) \circ \alpha^a
  = \alpha^{a + r} \circ j_!(g^{a,r})
 \colon
 j_!(\sh{M} \tensor f^* \sh{L}^a)
  \to j_*(\sh{M} \tensor f^* \sh{L}^{a + r}).
\end{equation*}
We will use the following self-evident properties of the $g^{a,r}$:

\begin{theorem}{lem}{(a,r) factorization}
 The $g^{a,r}$ satisfy:
 \begin{enumerate}
  \item \label{en:alternating}
  When $a + r \geq 0$, we have $g^{a, r} \circ g^{a + r, -r} = (1 - t)^{\abs{r}}$.

  \item \label{en:transitivity}
  When $r$ and $s$ have the same sign and $a + r + s \geq 0$, we have $g^{a, r + s} = g^{a + r, s}
  \circ g^{a, r}$.

  \item \label{en:images}
  Let $r \geq 0$, $a \geq r$; then we have:
  \begin{align*}
   \on{ker}(1 - t)^r =
   \on{ker}(g^{a,-r}_{\sh{M}}) \cong
   \sh{M} \tensor f^* \sh{L}^r,
   &&
   \on{im}(1 - t)^r =
   \on{im}(g^{a - r,r}_{\sh{M}}) \cong
   \sh{M} \tensor f^* \sh{L}^{a - r}.
  \end{align*}
 \end{enumerate}
 Finally, by \ref{uniformity} and (\plainref*{en:images}), for $r \geq 0$, $(1 - t)^{N + r}$
 annihilates $\on{ker}(\alpha^{a,-r})$ and $\on{coker}(\alpha^{a,r})$. \qed
\end{theorem}

From now on, we will assume $r \geq 0$.

\begin{theorem}{prop}{stable kernel and cokernel}
 For $a \gg 0$, the natural maps $j_!(g^{a, 1})$ and $j_*(g^{a + r,-1})$ respectively
 induce isomorphisms
 \begin{align*}
  \on{ker}(\alpha^{a,-r}) \xrightarrow{\sim} \on{ker}(\alpha^{a + 1, -r}) &&
  \on{coker}(\alpha^{a,r}) \xrightarrow{\sim} \on{coker}(\alpha^{a - 1, r})
 \end{align*}
 and $j_!(g^{a,r})$ and $j_*(g^{a + r,-r})$ induce isomorphisms
 \begin{align*}
  \on{ker}(\alpha^{a,r}) \xrightarrow{\sim} \on{ker}(\alpha^{a + r}) &&
  \on{coker}(\alpha^{a + r}) \xrightarrow{\sim} \on{coker}(\alpha^{a + r,-r})
 \end{align*}
\end{theorem}

\begin{proof}
 Using the maps $j_!(g^{a,1})$ and $j_*(g^{a - r, 1})$ we get a square which, using \ref{(a,r)
 factorization}(\plainref{en:alternating},\plainref{en:transitivity}), we verify is commutative:
 \begin{equation*}
  \xymatrix@C+2em{
   j_!(\sh{M} \tensor f^* \sh{L}^a)
    & j_* (\sh{M} \tensor f^* \sh{L}^{a - r}) \\
   j_!(\sh{M} \tensor f^* \sh{L}^{a + 1})
    & j_* (\sh{M} \tensor f^* \sh{L}^{a - r + 1})
   \ar"1,1";"1,2"^{\alpha^{a,-r}}
   \ar"1,2";"2,2"^{j_*(g^{a - r, 1})}
   \ar"1,1";"2,1"_{j_!(g^{a,1})}
   \ar"2,1";"2,2"^{\alpha^{a + 1, -r}}
  }
 \end{equation*}
 showing that $j_!(g^{a,1})$ induces a map on kernels.  Since it is injective, we get a long
 sequence of inclusions of kernels:
 \begin{equation*}
  \cdots \subset \on{ker} \alpha^{a - 1,-r} \subset \on{ker} \alpha^{a,-r}
         \subset \on{ker} \alpha^{a + 1,-r} \subset \cdots.
 \end{equation*}
 By \ref{(a,r) factorization}, each kernel is annihilated by $(1 - t)^{N + r}$,
 whose kernel is (for $a \geq N + r$) the perverse sheaf $j_! (\sh{M} \tensor f^* \sh{L}^{N +
 r})$; thus, this sequence is contained in this sheaf.  Since perverse sheaves are noetherian,
 this chain must have a maximum, so the kernels stabilize.  For the cokernels, we apply \ref{eq:g
 duality} to the argument of \ref{uniformity}.  (One can also argue directly using the artinian
 property of perverse sheaves.)

 For the second statement concerning kernels, since $(1 - t)^N$ annihilates
 $\on{ker}(\alpha^{a + r})$, for $a \geq N$ it is contained in $\on{im}(g^{a,r})$, and therefore by
 definition in $\on{ker}(\alpha^{a,r})$.  The statement on cokernels is again obtained by
 dualization and \ref{eq:g duality}.  (A direct argument employing a diagram chase is also
 possible, using the fact that $(1 - t)^N \on{coker}(\alpha^{a + r}) = 0$.)
\end{proof}

Departing slightly from Beilinson's notation, we denote these stable kernels and cokernels $\on{ker}
\alpha^{\infty, -r}$ and $\on{coker} \alpha^{\infty, r}$ for $r \geq 0$; when $r = 0$ we drop it.

\begin{theorem}{prop}{kernel equals cokernel}
 There is a natural isomorphism $\on{ker} \alpha^{\infty, -r} \xrightarrow{\sim} \on{coker}
 \alpha^{\infty, r}$.
\end{theorem}

\begin{proof}
 Consider the map of short exact sequences for any $a$ and any $b \geq r$ (to eliminate clutter we
 have not written the superscripts on the maps $g$):
 \begin{equation*}
  \xymatrix{
   0
    & {j_!(\sh{M} \tensor f^* \sh{L}^a)}
    & {j_!(\sh{M} \tensor f^* \sh{L}^{a + b})}
    & {j_!(\sh{M} \tensor f^* \sh{L}^b)}
    & 0 \\
   0
    & {j_*(\sh{M} \tensor f^* \sh{L}^{a + r})}
    & {j_*(\sh{M} \tensor f^* \sh{L}^{a + b})}
    & {j_*(\sh{M} \tensor f^* \sh{L}^{b - r})}
    & 0
   \ar"1,1";"1,2"
   \ar"1,2";"1,3"^-{j_!(g)}
   \ar"1,3";"1,4"^-{j_!(g)}
   \ar"1,4";"1,5"
   \ar"2,1";"2,2"
   \ar"2,2";"2,3"^-{j_*(g)}
   \ar"2,3";"2,4"^-{j_*(g)}
   \ar"2,4";"2,5"
   \ar"1,2";"2,2"^{\alpha^{a,r}}
   \ar"1,3";"2,3"^{\alpha^{a + b}}
   \ar"1,4";"2,4"^{\alpha^{b, -r}}
  }
 \end{equation*}
 By the snake lemma, we have an exact sequence of kernels and cokernels:
 \begin{multline}
  \label{eq:snake sequence}
  0 \to \on{ker}(\alpha^{a,r}) \to \on{ker}(\alpha^{a + b}) \to \on{ker}(\alpha^{b, -r})
   \xrightarrow{\gamma^{a,b;r}} \on{coker}(\alpha^{a,r}) \\
   \to \on{coker}(\alpha^{a + b}) \to \on{coker}(\alpha^{b, -r}) \to 0.
 \end{multline}
 If $a, b \gg 0$, then the first and last maps are, by the second part of \ref{stable kernel and
 cokernel}, isomorphisms.  Therefore $\gamma^{a,b;r}$ is an isomorphism.  Since the long exact
 sequence of cohomology \ref*{eq:snake sequence} is natural, we see that $\gamma^{a,b;r}$ is
 independent of $a$ and $b$ in the sense of the proposition:
 \begin{align*}
  \gamma^{a,b + 1;r} \circ j_!(g^{b,1}) = \gamma^{a,b;r}
  &&
  j_*(g^{a + 1,-1}) \circ \gamma^{a + 1,b;r} = \gamma^{a,b;r}
 \end{align*}
 where the requisite commutative diagrams are produced using \ref{(a,r)
 factorization}(\plainref{en:alternating},\plainref{en:transitivity}).  For the same reason,
 $\gamma^{a,b;r}$ is a natural transformation between the two functors
 \begin{equation*}
  \sh{M} \mapsto \on{ker}(\alpha^{b,-r}_{\sh{M}}),
  \qquad
  \sh{M} \mapsto \on{coker}(\alpha^{a,r}_{\sh{M}}),
 \qedhere
 \end{equation*}
\end{proof}

Because they are equal, we will give a single name $\Pi^r_f(\sh{M}) = \on{ker}(\alpha^{\infty, -r})
\cong \on{coker}(\alpha^{\infty, r})$ to the stable kernel and cokernel.  These are
thus exact functors, and by definition of $\alpha^{a,r}$ and \ref{eq:g duality}, they commute with
duality: $\DD \Pi_f^r(\sh{M}) \cong \Pi_f^r(\DD \sh{M})$.  From \ref{nearby cycles construction} we
conclude:

\begin{theorem}{cor}{nearby cycles computations}
 For $a \gg 0$ we have $\on{ker}(\alpha^a) \cong \Psiun_f(\sh{M}) \cong
 \on{coker}(\alpha^a)$, and thus an isomorphism \begin{equation*}
  \DD \Psiun_f(\sh{M}) \cong \Psiun_f(\DD\sh{M})
 \end{equation*}
 which is natural in the perverse sheaf $\sh{M}$.  A more effective, equivalent construction is
 obtained as follows: suppose $(1 - t)^N$ annihilates $\Psiun_f(\sh{M})$.  Then we have by
 \ref{eq:kernel and cokernel}:
 \begin{equation*}
  \Psiun_f(\sh{M})
    = i^* j_{!*} (\sh{M} \tensor f^* \sh{L}^N)[-1]
    = i^! j_{!*} (\sh{M} \tensor f^* \sh{L}^N)[1].
 \end{equation*}
 Conversely, if these equations hold, then of course $(1 - t)^N$ annihilates $\Psiun_f(\sh{M})$.
 \qed
\end{theorem}

\section{Vanishing cycles and gluing}
\label{sec:gluing}
We will refer to $\Pi_f^1$ as $\Xiun_f$, which Beilinson calls the ``maximal extension functor'' and
denotes without the superscript.  Although there is no independent, non-unipotent analogue, we have
chosen to use this notation to match that for the nearby and (upcoming) vanishing cycles functors,
which do have such analogues.

\begin{theorem}{prop}{xi sequences}
 There are two natural exact sequences exchanged by duality and $\sh{M} \leftrightarrow \DD \sh{M}$:
 \begin{gather*}
  0 \to j_!(\sh{M}) \xrightarrow{\alpha_-} \Xiun_f(\sh{M})
    \xrightarrow{\beta_-} \Psiun_f(\sh{M}) \to 0 \\
  0 \to \Psiun_f(\sh{M}) \xrightarrow{\beta_+} \Xiun_f(\sh{M})
    \xrightarrow{\alpha_+} j_*(\sh{M}) \to 0,
 \end{gather*}
 where $\alpha_+ \circ \alpha_- = \alpha$ and $\beta_- \circ \beta_+ = 1 - t$.
\end{theorem}

\begin{proof}
 These sequences are, respectively, the last and first halves of \ref{eq:snake sequence}.  For the
 first one, take $b = r$, so $g^{b,-r} = 0$ and therefore $\alpha^{b,-r} = 0$; we get an exact
 sequence
 \begin{equation*}
  0 \to \on{ker}(\alpha^{a,r})
    \to \on{ker}(\alpha^{a + r})
    \to j_!(\sh{M} \tensor f^* \sh{L}^r)
    \xrightarrow{\gamma^{a,r;r}}
        \on{coker}(\alpha^{a,r})
    \to \on{coker}(\alpha^{a + r})
    \to 0.
 \end{equation*}
 For $a \gg 0$, by the second part of \ref{stable kernel and cokernel}, the first map is an
 isomorphism, and for $r = 1$ we obtain the first short exact sequence from the remaining three
 terms above.  For the second short exact sequence, we apply the same reasoning to \ref*{eq:snake
 sequence} with $a = 0$ and then $r = 1$, with $b \gg 0$:
 \begin{equation*}
  0 \to \on{ker}(\alpha^b)
    \to \on{ker}(\alpha^{b,-r})
    \xrightarrow{\gamma^{0,b;r}}
        j_*(\sh{M} \tensor f^* \sh{L}^r)
    \to \on{coker}(\alpha^b)
    \to \on{coker}(\alpha^{b,-r})
    \to 0.
 \end{equation*}
 It is obvious from these constructions and \ref{eq:g duality} that the two short exact sequences
 are exchanged by duality.  To show that $\alpha_+ \circ \alpha_- = \id$ and $\beta_- \circ \beta_+
 = 1 - t$, we identify these maps in the above sequences and rewrite the claims as:
 \begin{align*}
  \bigl(\gamma^{0,b;1} \circ (\gamma^{a,b;1})^{-1} \circ \gamma^{a,1;1}\bigr)\bigr|_U = \id,
  &&
  \gamma^{a,b;1}|_{\on{ker}(\alpha^b)} \bmod \on{im}(\alpha^{a + 1}) = (1 - t) \gamma^{a + 1, b;0}.
 \end{align*}
 For both, we use the fact that since $\alpha^a|_U = \id$, we have $\gamma^{a,b;r}|_U = (g^{a +
 b,-a} \circ g^{a + r, b - r})^{-1}$, as constructed in the familiar proof of the snake lemma, with
 the inverse interpreted as a multi-valued pullback.  Then the claims are equivalent to
 \begin{align*}
  g^{a + b, -a} \circ g^{a + 1, b - 1} &= g^{1, b - 1} \circ g^{a + 1, -a}
  \\
  g^{a + b + 1, -a - 1} \circ g^{a + 1, b} &= (1 - t) g^{a + b, -a} g^{a + 1, b - 1}
 \end{align*}
 which follow from \ref{(a,r) factorization}(\plainref{en:alternating},\plainref{en:transitivity}).
\end{proof}

The remainder of the paper is simply what Beilinson calls ``linear algebra'' (one might argue that
this has already been the case for most of the preceding).  Take $\sh{M} = j^* \sh{F}$ for a
perverse sheaf $\sh{F} \in \cat{M}(X)$ in the above exact sequences.  From the maps in these two
 sequences we can form a complex:
\begin{equation}
 \label{eq:phi complex}
  j_! j^* \sh{F} \xrightarrow{(\alpha_-, \gamma_-)}
 \Xiun_f (j^* \sh{F}) \oplus \sh{F} \xrightarrow{(\alpha_+, -\gamma_+)}
 j_* j^* \sh{F},
\end{equation}
where $\map{\gamma_-}{j_!j^*(\sh{F})}{\sh{F}}$ and $\map{\gamma_+}{\sh{F}}{j_*j^* (\sh{F})}$ are
defined by the left- and right-adjunctions $(j_!, j^*)$ and $(j^*, j_*)$ and the property that
$j^*(\gamma_-) = j^*(\gamma_+) = \id$.

\begin{theorem}{prop}{vanishing cycles}
 The complex \ref{eq:phi complex} is in fact a complex; let $\Phiun_f(\sh{F})$ be its cohomology
 sheaf.  Then $\Phiun_f$ is an exact functor $\cat{M}(X) \to \cat{M}(Z)$, and there are maps
 $u,v$ such that $v \circ u = 1 - t$ as in the following diagram:
 \begin{equation*}
  \Psiun_f (j^* \sh{F}) \xrightarrow{u} \Phiun_f(\sh{F}) \xrightarrow{v} \Psiun_f (j^* \sh{F}).
 \end{equation*}
\end{theorem}

\begin{proof}
 That \ref*{eq:phi complex} is a complex amounts to showing that $\gamma_+ \circ \gamma_- = \alpha
 = \alpha_+ \circ \alpha_-$, which is true by definition of the $\gamma_{\pm}$ and adjunction.  To
 show that $\Phiun_f$ is exact, suppose we have $0 \to \sh{F}_1 \to \sh{F}_2 \to \sh{F}_3 \to 0$, so
 that we get a short exact sequence of complexes
 \begin{equation*}
  0 \to C^\bullet(\sh{F}_1) \to C^\bullet(\sh{F}_2) \to C^\bullet(\sh{F}_3) \to 0,
 \end{equation*}
 where by $C^\bullet(\sh{F})$ we have denoted the complex \ref*{eq:phi complex} padded with zeroes
 on both sides.  Note that since $\alpha_-$ is injective and $\alpha_+$ surjective,
 $C^\bullet(\sh{F})$ fails to be exact only at the middle term.  Therefore we have a long exact
 sequence of cohomology sheaves:
 \begin{equation*}
  \cdots 
  (0 = H^{-1} C^\bullet(\sh{F}_3))
  \to \Phiun_f(\sh{F}_1)
  \to \Phiun_f(\sh{F}_2)
  \to \Phiun_f(\sh{F}_3)
  \to (0 = H^1(C^\bullet(\sh{F}_1)))
  \cdots
 \end{equation*}
 which shows that $\Phiun_f$ is functorial and an exact functor.
 
 If we apply $j^*$ to \ref{eq:phi complex}, it becomes simply (with $j^* \sh{F} = \sh{M}$)
 \begin{equation*}
  \sh{M} \xrightarrow{(\id,\id)} \sh{M} \oplus \sh{M}
      \xrightarrow{(\id,-\id)} \sh{M}
 \end{equation*}
 which is actually exact, so $j^* \Phiun_f(\sh{F}) = 0$; i.e.\ $\Phiun_f(\sh{F})$ is supported on
 $Z$. Finally, to define $u$ and $v$, let $\map{\on{pr}}{\Xiun_f (j^* \sh{F}) \oplus \sh{F}}{\Xiun_f
 (j^* \sh{F})}$, and set $u = (\beta_+, 0)$ in coordinates, and $v = \beta_- \circ \on{pr}$.  Since
 $\beta_- \circ \alpha_- = 0$, $v$ factors through $\Phiun_f(\sh{F})$, and we have $v \circ u =
 \beta_- \circ \beta_+ = 1 - t$ by \ref{xi sequences}.
\end{proof}

Define a \emph{vanishing cycles gluing data} for $f$ to be a quadruple $(\sh{F}_U, \sh{F}_Z, u, v)$
as in \ref{vanishing cycles}; for any $\sh{F} \in \cat{M}(X)$, the quadruple $F_f(\sh{F}) =
(j^* \sh{F}, \Phiun_f(\sh{F}), u, v)$ is such data.  Let $\cat{M}_f(U,Z)$ be the category of gluing
data; then $\map{F_f}{\cat{M}(X)}{\cat{M}_f(U,Z)}$ is a functor.  Conversely, given a vanishing
cycles data
\begin{equation*}
 \Psiun_f(\sh{F}_U) \xrightarrow{u} \sh{F}_Z \xrightarrow{v} \Psiun_f(\sh{F}_U),
\end{equation*}
we can form the complex
\begin{equation}
 \label{eq:inverse gluing diagram}
 \Psiun_f(\sh{F}_U) \xrightarrow{(\beta_+,u)} \Xiun_f(\sh{F}_U) \oplus \sh{F}_Z
 \xrightarrow{(\beta_-,-v)} \Psiun_f(\sh{F}_U)
\end{equation}
since $v \circ u = 1 - t = \beta_- \circ \beta_+$, and let $G_f(\sh{F}_U, \sh{F}_Z, u, v)$ be its
cohomology sheaf.

Beilinson gives an elegant framework for proving the equivalence of \ref{eq:phi complex} and
\ref{eq:inverse gluing diagram} in \cite{this}*{Appendix}.  Rather than proving \ref{gluing}
directly, we present his technique (with slightly modified terminology).

\begin{theorem}{defn}{diads}
 Let a \emph{diad} be a complex of the form
 \begin{equation*}
  D^\bullet = \Bigl(\sh{F}_L \xrightarrow{L\, =\, (a_L, b_L)} \sh{A} \oplus \sh{B}
   \xrightarrow{R\, =\, (a_R, b_R)} \sh{F}_R\Bigr)
 \end{equation*}
 in which $a_L$ is injective and $a_R$ is surjective (so it is exact on the ends).  Let the category
 of diads be denoted $\cat{M}_2$.  Let a \emph{triad} be a short exact sequence of the form
 \begin{equation*}
  S = \Bigl(0 \to \sh{F}_- \xrightarrow{(c_-, d_-^1, d_-^2)} \sh{A} \oplus \sh{B}^1 \oplus \sh{B}^2
    \xrightarrow{(c_+, d_+^1, d_+^2)} \sh{F}_+ \to 0\Bigr)
 \end{equation*}
 in which both $\map{(c_-, d_-^i)}{\sh{F}_-}{\sh{A} \oplus \sh{B}^i}$ are injections and both
 $\map{(c_+, d_+^i)}{\sh{A} \oplus \sh{B}^i}{\sh{F}_+}$ are surjections.  Let the category of triads
 be denoted $\cat{M}_3$; it has a \emph{reflection functor} $\map{r}{\cat{M}_3}{\cat{M}_3}$
 which invokes the natural symmetry $1 \leftrightarrow 2$, and is an involution.
\end{theorem}

We can define a map $\map{T}{\cat{M}_2}{\cat{M}_3}$ by setting
\begin{equation*}
 T(D) = \Bigl(0 \to \on{ker}(R) \xrightarrow{(\iota_A, \iota_B, h)}
                    \sh{A} \oplus \sh{B} \oplus H(D^\bullet) \xrightarrow{(\pi_A, \pi_B, -k)}
                    \on{coker}(L) \to 0 \Bigr),
\end{equation*}
where the natural inclusion/projection (resp. projection/inclusion) are called:
\begin{align*}
 \on{ker}(R)
  \xrightarrow{\iota = (\iota_A, \iota_B)}
 \sh{A} \oplus \sh{B}
  \xrightarrow{\pi = (\pi_A, \pi_B)}
 \on{coker}(L),
 &&
 \on{ker}(R)
  \xrightarrow{h}
 H(D^\bullet)
  \xrightarrow{k}
 \on{coker}(L)
\end{align*}
(note $\pi \circ \iota = k \circ h$).  We define the inverse $T^{-1}$ by the formula
\begin{equation*}
 T^{-1}(S) = \Bigl( \on{ker}(d_-^2) \xrightarrow{(c_-, d_-^1)} \sh{A} \oplus \sh{B}^1 \to
\on{coker}(c_-, d_-^1) \Bigr).
\end{equation*}

\begin{theorem}{lem}{diad-triad equivalence}
 The functors $T, T^{-1}$ are mutually inverse equivalences of $\cat{M}_2$ with $\cat{M}_3$.
\end{theorem}

\begin{proof}
 Before beginning the verification of the many necessary facts, we observe that the property of a
 sequence $S$ as above being in $\cat{M}_3$ is equivalent to the following diagram being
 cartesian
 \begin{equation}
  \label{eq:cartesian equivalent}
  \vcenter{
  \xymatrix@C+2em{
    {\sh{F}_-} & {\sh{B}^i} \\
    {\sh{A}} \oplus \sh{B}^{3 - i} & {\sh{F}_+}
    \ar"1,1";"1,2"^-{d_-^i}
    \ar"1,2";"2,2"^{-d_+^i}
    \ar"1,1";"2,1"_{(c_-, d_-^{3 - i})}
    \ar"2,1";"2,2"_-{(c_+, d_+^{3 - i})}
   }
   }
 \end{equation}
 and the following smaller sequence being exact
 \begin{equation}
  \label{eq:small ses}
  0 \to \sh{F}_-               \xrightarrow{(c_-, d_-^i)}
        \sh{A} \oplus \sh{B}^i \xrightarrow{(c_+, d_+^i)}
        \sh{F}_+ \to 0.
 \end{equation}
 for $i = 1,2$.  Indeed, for \ref*{eq:cartesian equivalent}, the diagram is cartesian if and only
 if $S$ is exact in the middle, and for \ref*{eq:small ses}, the arrows are respectively injective
 and surjective by hypothesis if $S \in \cat{M}_3$, while exactness in the middle follows from
 that of $S$.  For readability, we continue the proof as several sub-lemmas.
 
 \medskip
 \begin{paragraph}{$T(D^\bullet) \in \cat{M}_3$:}
  It is easily verified that $T(D^\bullet)$ is an exact sequence. To see that $(\pi_A, -k)$ is
  surjective and $(\iota_A, h)$ injective, we consider the cartesian diagrams
  \begin{align*}
   \xymatrix@C+2em{
    {\sh{A} \oplus H(D^\bullet)} & {\on{coker}(L)} \\
    {\sh{A}} & {\sh{F}_R}
    \ar"1,1";"1,2"^-{(\pi_A, -k)}
    \ar"1,2";"2,2"
    \ar"1,1";"2,1"
    \ar"2,1";"2,2"^{a_R}
   }
   &&
   \xymatrix@C+2em{
    {\on{ker}(R)} & {\sh{A} \oplus H(D^\bullet)} \\
    {\sh{F}_L} & {\sh{A}}
    \ar"1,1";"1,2"^-{(\iota_A,h)}
    \ar"1,2";"2,2"
    \ar"1,1";"2,1"
    \ar"2,1";"2,2"^{a_L}
   }
  \end{align*}
  and use that $a_R$ is surjective and $a_L$ is injective.  Likewise, $(\pi_B, -k)$ is
  surjective and $(\iota_B, h)$ is injective.
 \end{paragraph}
 
 \medskip
 \begin{paragraph}{$T^{-1}(S) \in \cat{M}_2$:}
  Clearly, $T^{-1}(S)$ is a complex, since the sequence $\sh{F}_1 \to \sh{A} \oplus \sh{B}^1 \to
  \on{coker}(c_-, d_-^1)$ is exact (hence a complex).  Since $(c_-, d_-^2)$ is injective by
  hypothesis, $c_-|_{\on{ker}(d_-^2)}$ is injective.  We must show that $\sh{A} \to
  \on{coker}(c_-, d_-^1)$ is surjective, where by \ref{eq:small ses} with $i = 1$ we have
  $\on{coker}(c_-, d_-^1) = \sh{F}_+$;  consider the diagram
  \begin{equation*}
   \xymatrix@C+1em{
    {\sh{A} \oplus \sh{F}_-} & {\sh{A}} \\
    {\sh{A} \oplus \sh{B}^1} & {\sh{F}_+}
    \ar"1,1";"1,2"^-{(\id, -c_-)}
    \ar"1,2";"2,2"^{c_+}
    \ar"1,1";"2,1"_{\id \oplus d_-^1}
    \ar"2,1";"2,2"^-{(c_+, d_+^1)}
   }
  \end{equation*}
  which is cartesian by exactness of \ref*{eq:small ses}. Since the bottom arrow is a surjection,
  for $c_+$ to be a surjection it suffices to show that the left arrow is.  By
  \ref{eq:cartesian equivalent} with $i = 1$, $d_-^1$ is a surjection since the bottom arrow there
  is a surjection by  hypothesis.
 \end{paragraph}
 
 \medskip
 \begin{paragraph}{$T^{-1} \circ T \cong \id$:}
  Its $\sh{F}_L$ is $\on{ker}(h) = \on{im}(L) = \sh{F}_L$; its $\sh{A}$ and $\sh{B}$ are indeed
  $\sh{A}$ and $\sh{B}$, and its $\sh{F}_R$ is $\on{coker}(\iota) = \sh{F}_R$; one checks quickly
  that the maps are right as well.
 \end{paragraph}
 
 \medskip
 \begin{paragraph}{$T \circ T^{-1} \cong \id$:}
  Its $\sh{F}_-$ is $\on{ker}(\sh{A} \oplus \sh{B}^1 \to \on{coker}(c_-, d_-^1)) = \sh{F}_-$ since
  $(c_-, d_-^1)$ is an injection; its $\sh{A}$ and $\sh{B}^1$ are obviously the original $\sh{A}$
  and $\sh{B}^1$.  The small sequence \ref{eq:small ses} with $i = 1$ shows that $\sh{F}_+$ is
  correct as well.  Finally, for $\sh{B}^2$, we must show that $\sh{F}_-/\on{ker}(d_-^2) =
  \sh{B}^2$, or in other words, that $d_-^2$ is surjective, which follows from \ref{eq:cartesian
  equivalent} with $i  = 2$. \qedhere
 \end{paragraph}
\end{proof}

Clearly, both of the complexes \ref{eq:phi complex} and \ref{eq:inverse gluing diagram} are diads.
Comparing them, we find that the construction of the latter is given by:

\begin{theorem}{cor}{diad reflection}
 The reflection functor on a diad is the complex
 \begin{equation*}
  r(D^\bullet) = \Bigl( \on{ker}(a_R) \xrightarrow{(a_L',b_L')} \sh{A} \oplus H(D^\bullet)
                       \xrightarrow{(a_R', b_R')} \on{coker}(a_L) \Bigr),
 \end{equation*}
 where $a_L'$ is the natural inclusion and $a_R'$ the natural projection, $b_L' = h \circ (a_L',
 0)$, and $b_R'$ factors $-k$ through $\on{coker}(a_L) \subset \on{coker}(L)$.
\end{theorem}

\begin{proof}
 That is, $T^{-1} r T(D^\bullet) = r(D^\bullet)$ as defined above. We need to show that
 $\on{ker}(\iota_B) = \on{ker}(a_R)$ and $\on{coker}(\iota_A, h) = \on{coker}(a_L)$, and prove the
 identities of the morphisms.  The first is easily verified directly, considering both as subobjects
 of $\sh{A} \oplus \sh{B}$, while for the second, we assert that the map
 \begin{equation*}
  \map{(\id,0)}{\sh{A}}{\sh{A} \oplus H(D^\bullet)}
 \end{equation*}
 induces the desired isomorphism from the latter to the former.  To show that it identifies
 $\on{im}(a_L)$ with $\on{im}(\iota_A,h)$, it suffices to check that the following diagram is
 cartesian:
 \begin{equation*}
  \xymatrix@C+1em{
   {\sh{F}_L} & {\sh{A}} \\
   {\on{ker}(R)} & {\sh{A} \oplus H(D^\bullet)}
   \ar"1,1";"1,2"^-{a_L}
   \ar"1,2";"2,2"^{(\id,0)}
   \ar"1,1";"2,1"_L
   \ar"2,1";"2,2"_-{(\iota_A,h)}
  }
 \end{equation*}
 which follows from the definition of $H(D^\bullet) = \on{ker}(R)/\on{im}(R)$.  The identities of
 $a_L'$, $b_L'$, and $a_R'$ are clear from these constructions, while for $b_R'$ it is fastest to
 chase the above diagram.
\end{proof}

\begin{theorem}{thm}{gluing}
 The gluing category $\cat{M}_f(U,Z)$ is abelian; $\map{F_f}{\cat{M}(X)}{\cat{M}_f(U,Z)}$ and
 $\map{G_f}{\cat{M}_f(U,Z)}{\cat{M}(X)}$ are mutually inverse exact functors, and so
 $\cat{M}_f(U,Z)$ is equivalent to $\cat{M}(X)$.
\end{theorem}

\begin{proof}
 That $\cat{M}_f(U,Z)$ is abelian amounts to proving that taking coordinatewise kernels and
 cokernels works.  That is, if we have $(\sh{M}, \sh{F}_Z, u, v)$ and $(\sh{M}', \sh{F}_Z', u',
 v')$ with maps $\map{a_U}{\sh{M}}{\sh{M}'}$, $\map{a_Z}{\sh{F}_Z}{\sh{F}_Z'}$ and such that the
 following diagram commutes:
 \begin{equation*}
  \xymatrix{
   \Psiun_f (\sh{M}) \ar[r]^-{u} \ar[d]_{\Psiun_f(a_U)}
     & \sh{F}_Z \ar[r]^-{v} \ar[d]^{a_Z}
     & \Psiun_f (\sh{M}) \ar[d]^{\Psiun_f(a_U)} \\
   \Psiun_f (\sh{M}') \ar[r]^-{u'} & \sh{F}_Z' \ar[r]^-{v'} & \Psiun_f (\sh{M}')
  }
 \end{equation*}
 then $(\on{ker} a_U, \on{ker} a_Z, \tilde{u}, \tilde{v})$ is a kernel for $(a_U, a_V)$, where
 $\tilde{u}$ and $\tilde{v}$ are induced maps; likewise for the cokernel; and we must show that
 $(a_U, a_V)$ is an isomorphism if and only if the kernel and cokernel vanish.  The maps $\tilde{u}$
 and $\tilde{v}$ are constructed from the natural sequence of kernels (or cokernels) in the above
 diagram, and the exactness of $\Psiun_f$, and once they exist it is obvious from the definition of
 morphisms in $\cat{M}_f(U,Z)$ that the desired gluing data is a kernel (resp.\ cokernel).  Since
 $\cat{M}(U)$ and $\cat{M}(Z)$ are abelian and kernels and cokernels are taken coordinatewise, the
 last claim follows.
 
 To show that $F_f$ and $G_f$ are mutually inverse, we interpret $\cat{M}(X)$ and $\cat{M}_f(U,Z)$
 as diad categories in the form given, respectively, by diagrams \ref{eq:phi complex} and
 \ref{eq:inverse gluing diagram}.  The reflection functor is given by \ref{diad reflection}; by
 \ref{xi sequences} and the definition of $\Phiun$, its value on \ref*{eq:phi complex} is that of
 the functor $F_f$.  For the same reason, its value on \ref*{eq:inverse gluing diagram} is that of
 $G_f$ interpreted as a complex of type \ref*{eq:phi complex} (the $\sh{F}$ term is what we have
 previously called the value of $G_f$).  Since the reflection functor is an involution, $G_f$ and
 $F_f$ are mutually inverse.
\end{proof}

\section{Comments}
\label{sec:comments}

We conclude with some musings on the theory exposited here.  In the previous version
\href{http://arxiv.org/abs/1002.1686v2}{arXiv:1002.1686v2} of these notes, we gave a substantially
different proof of \ref{kernel equals cokernel}, adhering closely to that given in \cite{this}*{Key
Lemma}.  As that proof may better illuminate the two-sided limit formalism which we also omit, the
curious reader is encouraged to consult it.

\subsection*{The vanishing cycles functor\texorpdfstring{ and $\Phiun_f$}{}}

The functor $\Phiun_f$, like $\Psiun_f$, has a familiar identity.

\begin{theorem}{thm}{vanishing cycles and phi}
 There is an isomorphism of functors $\Phiun_f \cong R\phiun_f[-1]$ and a natural
 distinguished triangle
 \begin{equation*}
  \Psiun_f(j^* \sh{F}) \xrightarrow{u} \Phiun_f(\sh{F}) \to i^* \sh{F} \to
 \end{equation*}
 isomorphic to that in \ref{eq:phiun}.
\end{theorem}

\begin{proof}
 According to the definition of $\Phiun_f$ in \ref{vanishing cycles}, we have a short exact sequence
 and, thus, a corresponding distinguished triangle of the same form:
 \begin{equation*}
  0 \to j_! j^* \sh{F} \to \on{ker}(\alpha_+, -\gamma_+) \to \Phiun_f(\sh{F}) \to 0.
 \end{equation*}
 Since $K = \on{ker}(\alpha_+, -\gamma_+) \subset \Xiun_f(j^* \sh{F}) \oplus \sh{F}$, there is a
 projection map $\map{\on{pr}}{K}{\sh{F}}$ commuting with the inclusion of $j_! j^* \sh{F}$.

 Now we apply the octahedral axiom of triangulated categories as given in \cite{BBD}*{(1.1.7.1)}:
 \newlength{\arrlen} \setlength{\arrlen}{0.75in}
 \begin{equation*}
  \xymatrix{
   \ar@{}"1,1";"1,1"*+{j_! j^* \sh{F}}="A"
   \ar"A";"A"+/va(60)\arrlen/*+{K}="B"
   \ar"B";"B"+/va(60)\arrlen/*+{\Phiun_f(\sh{F})}="C'"
   \ar"C'";"C'"+/va(-60)\arrlen/*+{i^* \sh{F}}="B'"
   \ar"B'";"B'"+/va(-60)\arrlen/*+{C}="A'"
   \ar@{}"A";"B'"|!{"B";"A'"}*{}="M"
   \ar"A";"M"*+{\sh{F}}="C"
   \ar"C";"A'" \ar"C";"B'" \ar"B";"C"^{\on{pr}} \ar"C'";"B'"
  }
 \end{equation*}
 where all the straight lines are distinguished triangles, both the (geometric) triangles are
 commutative, and the square commutes.  It is easy to see that $\on{pr}$ must be surjective because
 $\alpha_+$ is surjective; thus, since both $K$ and $\sh{F}$ are perverse, $C[-1]$ is also
 perverse, and so we have an exact sequence
 \begin{equation*}
  0 \to C[-1] \to K \xrightarrow{\on{pr}} \sh{F} \to 0.
 \end{equation*}
 But by definition, $\on{ker}(\on{pr}) = \on{ker}(\alpha_+) \oplus 0$, and therefore $C[-1] \cong
 \Psiun_f(j^* \sh{F})$.  Note that the inclusion then becomes the map $u$, as defined in the proof
 of \ref{vanishing cycles}.  Rotating the other triangle in the above octahedral diagram, we have
 \begin{equation*}
  \Psiun_f(j^* \sh{F}) \xrightarrow{u} \Phiun_f(\sh{F}) \to i^* \sh{F} \to.
 \end{equation*}
 Comparing with \ref{eq:phiun}, we find that $R\phiun_f(\sh{F})[-1] \cong \Phiun_f(\sh{F})$ is
 perverse. Conversely, starting from \ref*{eq:phiun} in place of the above triangle, we conclude by
 the octahedral axiom that $R\phiun_f(\sh{F})[-1]$ is the cohomology of \ref{eq:phi complex}, which
 admits a unique extension to a functor of $\sh{F}$ compatible with the octahedral diagram.
 Therefore, we conclude an isomorphism of functors $\Phiun_f \cong R\phiun_f[-1]$.
\end{proof}

\subsection*{The full nearby cycles functor\texorpdfstring{ $R\psi_f$}{}}

As Beilinson observes, the full nearby cycles functor $R\psi_f(\sh{M})$, for $\sh{M} \in
\cat{M}(U)$, can be recovered from $R\psiun_f$ as applied to variations of $\sh{M}$.  Here we must
assume that the field of coefficients is algebraically closed.

\begin{theorem}{lem}{nearby cycles decomposition}
 There exists a unique isomorphism of functors $\cat{D}(U) \to \cat{D}(Z)$
 \begin{equation}
  \label{eq:nearby cycles decomposition}
  R\psi_f = \bigoplus_{\lambda \in \C^*} R\psi_f^\lambda
 \end{equation}
 where for any constructible complex $A^\bullet_U$ on $U$, $\lambda - t$ is nilpotent on
 $R\psi_f^\lambda(A^\bullet_U)$.
\end{theorem}

\begin{proof}
 Simply pursue the line of reasoning in \ref{unipotent nearby cycles} but, since the field of
 coefficients is algebraically closed, produce the full Jordan decomposition rather than just the
 unipotent and non-unipotent parts.  The lemma can also be deduced from \cite{eigenspaces}*{Lemme
 3.2.5}, which applies to the Jordan decomposition of an endomorphism of any complex in the derived
 category.
\end{proof}

Let $\sh{L}_\lambda$ be the local system of rank $1$ on $\Gm$ with monodromy $\lambda$; then
clearly, we have $R\psi_f^\lambda(\sh{M}) = R\psiun_f(\sh{M} \tensor f^* \sh{L}_\lambda^{-1})
\tensor L_\lambda$, where $t$ acts as $\lambda$ on the one-dimensional vector space $L_\lambda$.
Substituting into \ref{eq:nearby cycles decomposition}, we obtain:
\begin{equation*}
 R\psi_f(\sh{M})
   = \bigoplus_\lambda \Psiun_f(\sh{M} \tensor f^*\sh{L}_\lambda^{-1}) \tensor L_\lambda.
\end{equation*}
Thus, \ref{nearby cycles computations} gives a procedure for computing the full nearby cycles
functor of perverse sheaves, and $R\psi_f[-1]$ sends perverse sheaves on $U$ to perverse sheaves on
$X$.

Using some general reasoning, we can extend the properties of $\Psiun_f = R\psiun_f[-1]$ from the
subcategory of perverse sheaves to the entire derived category.  To this end, let
$\map{T}{\cat{C}}{\cat{D}}$ be a triangulated functor between triangulated categories with
t-structures, and let the respective cores be the abelian categories $\cat{A}$, $\cat{B}$. We will
assume that the objects of $\cat{C}$ are \emph{bounded above}, meaning that $\cat{C} = \bigcup_{b
\in \Z} \cat{C}^{\leqslant b}$.

\begin{theorem}{lem}{truncation}
 Suppose $T$ is right t-exact and that $T \cat{A} \subset \cat{B}$; then $T$ is t-exact.
\end{theorem}

\begin{proof}
 We will show that $T$ commutes with all truncations.  Suppose we have an object $x \in
 \cat{C}^{\leqslant b}$, so that there is a distinguished triangle
 \begin{equation*}
  \tau^{< b} x \to x \to \tau^{\geqslant b} x \to
 \end{equation*}
 where by definition, $\tau^{\geqslant b} x = H^b(x)[-b] \in \cat{A}[-b]$.  By hypothesis on $T$, we
 have $T(x) \in \cat{D}^{\leqslant b}$, $T(\tau^{< b} x) \in \cat{D}^{< b}$, and $T(H^b x[-b]) \in
 \cat{B}[-b] \subset \cat{D}^{\geqslant b}$.  Since $T$ is triangulated, there is a triangle
 \begin{equation*}
  T(\tau^{< b} x) \to T(x) \to T(H^b x[-b]) \to
 \end{equation*}
 and therefore, by uniqueness of the truncation triangle, it must be that $T(\tau^{< b} x) = \tau^{<
 b} T(x)$.  This is under the hypothesis that $x \in \cat{C}^{\leqslant b}$; since then $\tau^{< b}
 x \in \cat{C}^{\leqslant b - 1}$ and since $\tau^{< b - 1} \tau^{< b} = \tau^{< b - 1}$, we can
 apply truncations-by-one repeatedly and conclude that for all $n$, we have $\tau^{\leqslant n} T(x)
 = T(\tau^{\leqslant n} x)$.
 
 Now suppose we have any $x$, and for any $n$ form the distinguished triangle
 \begin{equation*}
  \tau^{< n} x \to x \to \tau^{\geqslant n} \to
 \end{equation*}
 to which we apply $T$.  Since $T(\tau^{< n} x) = \tau^{< n} T(x)$, the cone of the resulting
 triangle
 \begin{equation*}
  \tau^{< n} T(x) \to T(x) \to T(\tau^{\geqslant n} x) \to
 \end{equation*}
 must be isomorphic to $\tau^{\geqslant n} T(x)$, by uniqueness of cones and the truncation triangle
 for $T(x)$.  Thus, $\tau^{\geqslant n} T(x) = T(\tau^{\geqslant n} x)$.  Since then $T$ commutes
 with all trunctions, it is \emph{a fortiori} t-exact.
\end{proof}

Take $T = R\psiun_f[-1]$; by \ref{nearby cycles are perverse}, it satisfies the hypothesis of
\ref{truncation}, and therefore we conclude:

\begin{theorem}{thm}{nearby cycles are exact}
 The functor $R\psi_f[-1]$ on the bounded derived category $\cat{D}^b(X)$ is t-exact for
 the perverse t-structure.  Likewise, $R\phi_f[-1]$ is t-exact.
\end{theorem}

\begin{proof}
 For the second statement, we must show that $R\phi_f[-1]$ is right t-exact and preserves perverse
 sheaves; the latter claim already follows from \ref{vanishing cycles and phi}.  For the former,
 we apply the long exact sequence to the triangle
 \begin{equation*}
  i^* \sh{F} \to R\psi_f(j^* \sh{F}) \to R\phi_f(\sh{F}) \to
 \end{equation*}
 We have $i^* \sh{F} \in \perv \cat{D}(X)^{[-1,0]}$ because of triangle \ref{eq:extension
 triangle}, and we already know that $R\psi_f[-1]$ is right t-exact, so the long exact sequence of
 perverse cohomology shows that $\perv H^i(R\phi_f \sh{F}) = 0$ when $i \geq 0$, as desired.
\end{proof}

We will not prove here that $R\psi_f[-1]$ commutes with Verdier duality.  This is significantly more
difficult since it necessitates enlarging the domain of a certain natural transformation (the map
$\gamma^{a,b;r}$ constructed in \ref{kernel equals cokernel}) from the core of the perverse
t-structure to the entire derived category. This involves the interaction with both objects and
morphisms:
\begin{itemize}
 \item The natural maps must be defined for all objects, not just those in $\cat{M}(U)$;
 \item The maps thus obtained must commute with all morphisms, not just those between objects of
 $\cat{M}(U)$.
\end{itemize}
To see why this is difficult, consider showing merely that the $\gamma^{a,b;r}$ (and their
translates) are natural with respect to maps of the form $\map{g}{\sh{M}}{\sh{N}[i]}$, with $i \in
\N$ and $\sh{M}, \sh{N} \in \cat{M}(U)$.  Note that the argument given for the naturality of
$\gamma^{a,b;r}$ is not valid in this context, since kernel and cokernel constructions in the
abelian category of perverse sheaves are not functorial in the entire derived category.

If $i = 1$, this is easy; we necessarily have $\on{Cone}(g) \in \cat{M}(U)[1]$, so
rotating the distinguished triangle gives a short exact sequence
\begin{equation*}
 0 \to \sh{N} \to \on{Cone}(g)[-1] \to \sh{M} \to 0.
\end{equation*}
Conversely, this sequence constructs the distinguished triangle $\sh{M} \to \sh{N}[1] \to
\on{Cone}(g)$ by the reverse procedure.  Then, applying $\DD \Psiun_f$ and $\Psiun_f \DD$ to the
sequence, we find by naturality of $\gamma^{a,b;r}$ that there is a commutative diagram of short
exact sequences, which implies that $\gamma^{a,b;r}$ is natural with respect to $g$.

The analogue of this argument for $i > 1$ would involve finding a sequence of the form
\begin{equation*}
 0 \xrightarrow{h_{i + 1} = 0} (\sh{N} = \sh{A}^{-(i + 1)})
   \xrightarrow{h_i} \sh{A}^{-i} \to \dots \to \sh{A}^{-1}
   \xrightarrow{h_0} (\sh{A}^0 = \sh{M})
   \to 0
\end{equation*}
representing $g$.  The manner in which such a sequence does represent such a map is clear; we get a
collection of short exact sequences representing maps:
\begin{align*}
 0 \to \on{coker}(h_{j + 1}) \to \sh{A}^{-j} \to \on{coker}(h_j) \to 0, &&
 \map{g_j}{\on{coker}(h_j)}{\on{coker}(h_{j + 1})[1]}
\end{align*}
(where $\on{coker}(h_{i + 1}) = \sh{N}$ and $\on{coker}(h_1) = \sh{M}$), and thus, by composition, a
map $\map{g}{\sh{M}}{\sh{N}[i]}$, as desired.  This is Yoneda's realization of $\on{Ext}^i(\sh{M},
\sh{N})$; it holds in the derived category of $\cat{M}(U)$.  It is, however, a nontrivial theorem,
proved in \cite{B}, that this is the same as $\cat{D}(U)$, and in fact it is describing the
morphisms that occupies the entirety of the work in that paper.  Of course, once we choose to cite
this result, it is a trivial consequence of \ref{nearby cycles computations} that $R\psi_f[-1]$
commutes with $\DD$, since it is then the derived functor of a self-dual \emph{exact} functor on
$\cat{M}(U)$.  Thus, we do not expect that there will be as elementary an argument as for the
perversity of nearby cycles.

In the recent preprint \cite{duality}, autoduality of the nearby cycles functor is proven in
complete generality in the complex analytic setting, and references are given there for prior
results and those in the algebraic setting.

\subsection*{The maximal extension functor\texorpdfstring{ $\Xiun_f$}{}}

We have used the term ``maximal extension functor'' without explanation (as did Beilinson), but
\ref{xi sequences} provides sufficient rationale: applying $i^*$ to the first one and $i^!$ to the
second one, the long exact sequence of perverse cohomology shows that $i^* \Xiun_f(\sh{M}) \cong
\Psiun_f(\sh{M}) \cong i^! \Xiun_f(\sh{M})$ are both perverse sheaves, which is as far out
(cohomologically) as they can be given that $i^*$ is right t-exact and $i^!$ is left t-exact.  This
should be compared with the defining property of the ``minimal extension'' $j_{!*}(\sh{M})$, that
$i^* j_{!*}(\sh{M})[-1]$ and $i^! j_{!*}(\sh{M})[1]$ are perverse, so that it has a minimal presence
on $X$ given that it extends $\sh{M}$.  The condition that $i^* \Xiun_f(\sh{M})$ and $i^!
\Xiun_f(\sh{M})$ are perverse does not uniquely characterize $\Xiun_f(\sh{M})$, as one could add any
perverse sheaf supported on $Z$ without changing it, but imposing \ref*{xi sequences} forbids such a
modification.  As we will see below, these sequences uniquely determine $\Xiun_f(\sh{M})$.

To do so, consider the pair of upper and lower ``caps'' of an octahedron:
\begin{align*}
 \underset{\text{\large(upper cap)}}{\xymatrix@C-1ex@R+2ex{
   {j_! \sh{M}} & & {j_* \sh{M}} \\
   & {\Xiun_f(\sh{M})} & \\
   {\Psiun_f(\sh{M})} & & {\Psiun_f(\sh{M})[1]}
   \ar"1,1";"1,3"^{\alpha}
   \ar"1,1";"2,2"_{\alpha_-}
   \ar"2,2";"1,3"_{\alpha_+}
   \ar"2,2";"3,1"_{\beta_-}
   \ar"3,1";"1,1"^{[1]}
   \ar"1,3";"3,3"
   \ar"3,3";"2,2"_{\beta_+[1]}
   \ar"3,3";"3,1"^{(1 - t)[1]}
   \ar@{}"2,2";"1,2"|(0.67){c}
   \ar@{}"2,2";"2,3"|(0.67){d}
   \ar@{}"2,2";"3,2"|(0.67){c}
   \ar@{}"2,2";"2,1"|(0.67){d}
  }
 }
 &&
 \underset{\text{\large(lower cap)}}{\xymatrix@C-1ex@R+2ex{
   {j_! \sh{M}} & & {j_* \sh{M}} \\
   & {i^* j_* \sh{M}\vphantom{\Xiun_f(\sh{M})}} & \\
   {\Psiun_f(\sh{M})} & & {\Psiun_f(\sh{M})[1]}
   \ar"1,1";"1,3"^{\alpha}
   \ar"1,3";"2,2"
   \ar"2,2";"1,1"^{[1]}
   \ar"3,1";"1,1"^{[1]}
   \ar"3,1";"2,2"
   \ar"1,3";"3,3"
   \ar"2,2";"3,3"
   \ar"3,3";"3,1"^{(1 - t)[1]}
   \ar@{}"2,2";"1,2"|(0.67){d}
   \ar@{}"2,2";"2,3"|(0.67){c}
   \ar@{}"2,2";"3,2"|(0.67){d}
   \ar@{}"2,2";"2,1"|(0.67){c}
  }
 }
\end{align*}
The triangles marked ``c'' are commutative and those marked ``d'' are distinguished; the arrows
marked $[1]$ have their targets (but not their sources) shifted by $1$.  The octahedral axiom states
that given any diagram of commutative and distinguished triangles as in (lower cap) we can
construct a diagram as in (upper cap) and vice versa (\cite{BBD}*{\S1.1.6}).  Using these diagrams,
we can derive \ref{eq:psiun} and \ref{xi sequences} from each other.  This idea is also present in
\cite{survey}*{\S5.7.2}.

\begin{theorem}{prop}{xi implies psiun}
 Suppose we have functors $\Xiun_f$ and $\Psiun_f$ from $\cat{M}(U)$ to $\cat{M}(X)$, where
 $\Psiun_f$ has a unipotent action of $\pi_1(\Gm)$, and satisfying \ref{xi sequences}.  Then
 \ref{eq:psiun} holds with $R\psiun_f = \Psiun_f[1]$.
\end{theorem}

\begin{proof}
 Given \ref*{xi sequences}, each exact sequence there corresponds to a unique distinguished triangle
 in $\cat{D}(X)$ with the same entries; these triangles appear in (upper cap), where the top and
 bottom maps are $\alpha$ and $(1 - t)[1]$ since the triangles containing them are commutative.  The
 octahedral axiom gives us (lower cap), and since the upper triangle is distinguished its cone (the
 middle term) must necessarily be $i^* j_* \sh{M}$ by \ref{eq:extension triangle}.  Therefore the
 bottom triangle is \ref*{eq:psiun}, as desired.  Note that all the interior maps in (lower cap)
 are uniquely determined, since they correspond to the kernels and cokernels of the maps $\alpha$
 and $1 - t$ of perverse sheaves in the long exact sequence of cohomology.
\end{proof}
 
\begin{theorem}{prop}{psiun implies xi}
 Given only the triangle \ref{eq:psiun}, both the functor $\Xiun_f$ and its extension classes in
 $\on{Ext}^1(\Psiun_f(\sh{M}), j_! \sh{M})$ and $\on{Ext}^1(j_* \sh{M}, \Psiun_f(\sh{M}))$ can
 be constructed with \ref{xi sequences} satisfied (except for the duality statement).  In
 particular, by \ref{xi implies psiun}, $\Xiun_f$ is uniquely determined by \ref*{xi sequences}.
\end{theorem}

\begin{proof}
 Given \ref*{eq:psiun}, since we have \ref{eq:extension triangle} canonically we can form all
 the vertices of (lower cap) and both distinguished triangles; the left and right maps are
 determined by the requirement that the triangles containing them be commutative.  The octahedral
 axiom gives us (upper cap) and $\Xiun_f(\sh{M})$, identified at first only as an element of
 $\cat{D}(X)$.  From \ref{nearby cycles are perverse} we know that $\Psiun_f(\sh{M})$ is perverse;
 then the long exact sequence of perverse cohomology associated to either distinguished triangle in
 (upper cap) shows that, in fact, $\Xiun_f(\sh{M})$ is perverse, and thus those triangles correspond
 to exact sequences as in \ref*{xi sequences}.  The equations $\alpha_+ \alpha_- = \alpha$ and
 $\beta_- \beta_+ = 1 - t$ can then be read off from the commutativity of the upper and lower
 triangles.  Since the vertical arrows come from (lower cap), these distinguished triangles are
 uniquely determined up to isomorphism fixing $j_{*,!}\sh{M}$ and $\Psiun_f(\sh{M})$, as desired.
\end{proof}

The identity of $\Xiun_f$ is somewhat mysterious, but can be made precise using the gluing
category.  These computations are also given in \cite{survey}*{Example 5.7.8}.

\begin{theorem}{prop}{glued functors}
 For any perverse sheaf $\sh{M} \in \cat{M}(U)$, we have the following correspondences via the
 gluing construction:
 \begin{align*}
  j_!(\sh{M})
   &= (\sh{M}, \Psiun_f(\sh{M}), \id, 1 - t)
  &
  j_{!*}(\sh{M})
   &= (\sh{M}, \on{im}(1 - t), 1 - t, \on{incl})
  \\
  j_*(\sh{M})
   &= (\sh{M}, \Psiun_f(\sh{M}), 1 - t, \id)
  &
  \Xiun_f(\sh{M})
   &= (\sh{M}, \Psiun_f(\sh{M} \tensor f^* \sh{L}^2), u, v);
 \end{align*}
 where $\map{\alpha}{j_!(\sh{M})}{j_*(\sh{M})}$ is the map $(\id, 1 - t)$ in the gluing category; in
 $j_{!*}(\sh{M})$, we mean $\on{im}(1 - t) \subset \Psiun_f(\sh{M})$; in
 $\Xiun_f(\sh{M})$, taking $\Psiun_f(\sh{M} \tensor f^* \sh{L}^a) = \Psiun_f(\sh{M}) \oplus
 \Psiun_f(\sh{M})$, we have $u = (\id, 1 - t)$ and $v = \on{pr}_2$.
\end{theorem}

\begin{proof}
 Using the triangle of \ref{vanishing cycles and phi}, we have
 \begin{equation*}
  \Psiun_f(j^* j_! \sh{M}) \to \Phiun_f(j_! \sh{M}) \to i^* j_!(\sh{M}) \to
 \end{equation*}
 and since $i^* j_! = 0$, we get an isomorphism $\Phiun_f(j_! \sh{M}) \cong \Psiun_f(\sh{M})$;
 dualizing, we have $\Phiun_f(j_* \sh{M}) \cong \Psiun_f(\sh{M})$ also.  Since $u$ is the first map
 in this triangle, under this identification we have $u = \id$, and therefore $v = 1 - t$ since $v
 \circ u = 1 - t$.  This gives the quadruple for $j_!(\sh{M})$; for $j_*(\sh{M})$, we dualize,
 since $u$ and $v$ are dual by their definition in Propositions \plainref{vanishing cycles} and
 \plainref{xi sequences}. That the natural map is given by $(\id, 1 - t)$ follows from the fact that
 this does define a map $j_!(\sh{M}) \to j_*(\sh{M})$ in the gluing category, and that its
 restriction to $U$ is the identity.

 For $j_{!*}(\sh{M})$, we use the fact that it is the image of the natural map
 $\map{\alpha}{j_!(\sh{M})}{j_*(\sh{M})}$; having already identified all the parties, this is clear
 from the quadruples just obtained.

 For the identification of $\Xiun_f(\sh{M})$, obviously, $v \circ u = 1 - t$; more importantly, $u$
 is injective and $v$ surjective. Then the pair of exact sequences in \ref{xi sequences} can be
 described on quadruples as being trivial over $U$, and over $Z$ the maps $\alpha_-$ and $\alpha_+$
 are described by the following maps of quadruples:
 \begin{equation*}
  \xymatrix{
   {j_!(\sh{M}):} & 
    {\Psiun_f(\sh{M})} & {\Psiun_f(\sh{M})} & {\Psiun_f(\sh{M})} \\
   {\Xiun_f(\sh{M}):} &
    {\Psiun_f(\sh{M})} & {\Psiun_f(\sh{M} \tensor f^*\sh{L}^2)} & {\Psiun_f(\sh{M})} \\
   {j_*(\sh{M}):} &
    {\Psiun_f(\sh{M})} & {\Psiun_f(\sh{M})} & {\Psiun_f(\sh{M})}
   \ar"1,1";"2,1"_{\alpha_-} \ar"2,1";"3,1"_{\alpha_+}
   \ar"1,2";"1,3"^-{\id} \ar"1,3";"1,4"^-{1 - t}
   \ar"2,2";"2,3"^-{u} \ar"2,3";"2,4"^-{v}
   \ar"3,2";"3,3"^-{1 - t} \ar"3,3";"3,4"^-{\id}
   \ar"1,2";"2,2"^{\id} \ar"2,2";"3,2"^{\id}
   \ar"1,3";"2,3"^{u} \ar"2,3";"3,3"^{v}
   \ar"1,4";"2,4"^{\id} \ar"2,4";"3,4"^{\id}
  }
 \end{equation*}
 We take $\beta_-$ and $\beta_+$ to be the maps whose $Z$-parts (the $U$-parts are zero) are:
 \begin{align*}
  \beta_-(y,z) = (1 - t)y - z
  &&
  \beta_+(x) = (x, 0).
 \end{align*}
 Then it is clear from the definitions of $u$ and $v$ that we obtain the sequences of \ref{xi
 sequences}; by the uniqueness part of \ref{psiun implies xi}, this uniquely determines
 $\Xiun_f(\sh{M})$, completing the proof.
\end{proof}

Since the entirety of \ref{sec:gluing} follows only from \ref*{xi sequences}, Propositions
\plainref*{xi implies psiun} and \plainref*{psiun implies xi} show that the constructions of
\ref{sec:unipotent cycles} are irrelevant for constructing the gluing functor.  Their purpose, as is
evident from the order we have chosen for the theorems, is to exhibit the autoduality of $\Psiun_f$
and $\Xiun_f$ (and, thus, $\Phiun_f$).  However, Beilinson's development has an aesthetic virtue
(over just using the above short proof of \ref*{xi sequences}): once \ref{nearby cycles are
perverse} is proven, the entire theory takes place within the abelian category of perverse sheaves. 
In addition, \ref{nearby cycles computations} is an ingeniously elementary, insightful, and more
useful definition of a functor whose actual definition is quite obscure.

\end{document}